%% file: decoupling-bounds.tex
\documentclass{article}

\usepackage[utf8]{inputenc}

\def\showauthornotes{0}
\def\showkeys{0}
\def\showdraftbox{1}
\def\widemargin{0}

\input{macros}

\newcommand{\trans}[1]{{#1}^{\mathsf{T}}}

\title{Simple Norm Bounds for Polynomial Random Matrices via Decoupling}
\author{
Madhur Tulsiani\thanks{{\tt TTIC}. {\tt madhurt@ttic.edu}. Supported by NSF grant CCF-1816372 and CCF-2326685.} 
\and
June Wu \thanks{{\tt University of Chicago}. {\tt jqw@uchicago.edu}. } 
}

\allowdisplaybreaks

\usepackage{tikz,pgf}
\usepackage{eufrak}
\usepackage{graphicx}
\newcommand{\mat}[1]{\mathbf{#1}}

\newcommand{\mA}{{\mat{A}}}

\newcommand{\mF}{{\mat{F}}}
\newcommand{\Esch}[2]{\EE \norm{#1}_{#2}^{#2}}
\newcommand{\EE}{\mathbb{E}}
\newcommand{\Etr}[1]{\EE \tr\left[#1\right]}
\newcommand{\T}{\intercal}
\newcommand{\resamp}[1]{\widetilde{#1}}

\begin{document}

\date{}

\maketitle

\thispagestyle{empty}

\begin{abstract}
We present a new method for obtaining norm bounds for random matrices, where each entry is a
low-degree polynomial in an underlying set of independent real-valued random variables.
Such matrices arise in a variety of settings in the analysis of spectral and optimization
algorithms, which require understanding the spectrum of a random matrix depending on data
obtained as independent samples.

Using ideas of decoupling and linearization from analysis, we show a simple way of expressing norm
bounds for such matrices, in terms of matrices of lower-degree polynomials corresponding to
derivatives. Iterating this method gives a simple bound with an elementary proof, which can recover
many bounds previously required more involved techniques. 
\end{abstract}

\newpage

\pagenumbering{roman}
\tableofcontents
\clearpage

\newpage
\pagenumbering{arabic}
\setcounter{page}{1}

\section{Introduction}
\input{intro.tex}

\section{Preliminaries and Notation}\label{sec:prelims}

Vectors are denoted by bold face lower case letters $\bfx, \bfy$ and $\bfz$. Deterministic matrices are denoted by bold face upper case letters $\bfA$ and $\bfB$. Matrix-valued functions are denoted by bold face upper case letters $\bfX$, $\bfY$, $\bfF$, $\bfM$, $\bfH$ and $\bfP$. 

\vspace{-10 pt}

\paragraph{Sets and indices.} 
For a set $S$, let $|S|$ denote the number of distinct elements in $S$. We denote $[n] = \{1,\dots, n\}$ for any positive integer $n$. For $\bfi := (i_1, \dots, i_m) \in [n]^m$, define
$$
\calT^m_n := \{(i_1, \dots, i_m): \forall j,k \in [m], ~ j \neq k \implies i_j \neq i_k\}. 
$$

\vspace{-10 pt}

\paragraph{Matrix norms.} 
Let $\R ^{d_1 \times d_2}$ be the space of all $d_1 \times d_2$ real matrices. Let $\mathbb{H}^d$ be the subspace of $ \R ^{d \times d}$ containing all Hermitian matrices. We write $\norm{\cdot}_2$ for the $\ell_2$ operator norm, $\norm{\cdot}_F$ for the Frobenius norm, and $\tr(\cdot)$ for the trace. 

\begin{definition} [Schatten norm] 
For $\bfA \in \R^{d_1 \times d_2}$ and $t \geq 1$, the Schatten $2t$-norm is defined as
$$
\norm{\bfA}_{2t} := (\tr (\trans{\bfA}\bfA)^t)^{1/2t}.
$$
\end{definition}

\vspace{-10 pt}

\paragraph{Polynomial random matrices.} 
Let $(\Omega, \cal F, \mu)$ be a probability space. Introduce the random vector $\bfx := (x_1, \dots, x_n) \in \Omega$ where $x_1, \dots, x_n$ are independent random variables. A polynomial random matrix ${\bf F}(\bfx)$ is a matrix-valued polynomial $${\bf F}: \Omega \rightarrow \R ^{d_1 \times d_2}.$$ We can construct a polynomial random matrix ${\bf F}(\bfx)$ by drawing a copy of $\bfx \sim \mu$. Note that despite $\bfx$ having independent entries, the entries of ${\bf F}(\bfx)$ can be dependent (or correlated) with each other. 

\begin{example}[Rademacher chaos of order 2] \label{example} We give an example of a degree-2 polynomial random matrices in 3 variables. Let $x_1, x_2, x_3$ be i.i.d. Rademacher random variables and $\{\bfB_\bfi\}_{\bfi\in \calT^2_3}$ be a double sequence of deterministic matrices (explicitly chosen here) of the same dimension.
\begin{equation} \label{not symmetric}
\sum_{\bfi \in \calT^2_3} \bfB_\bfi \cdot x_ix_j = x_1x_2\begin{bmatrix} 1 & 0 \\ 0 & 0 \end{bmatrix} + x_1x_3\begin{bmatrix} 0 & 0 \\ 0 & 1 \end{bmatrix} + x_2x_3\begin{bmatrix} 0 & 1 \\ 1 & 0 \end{bmatrix} = \begin{bmatrix} x_1x_2 & x_2x_3 \\ x_2x_3 & x_1x_3 \end{bmatrix}.   
\end{equation}
\end{example}

\begin{remark}
Unlike Wigner matrices whose entries are i.i.d., the polynomial random matrices can have dependencies among its entries.
\end{remark}

\begin{definition}[Permutation Symmetric Property]
Suppose $\bfF(\bfx)$ is a degree-d homogeneous multilinear polynomial random matrix, \ie
$$
\bfF(\bfx) =  \sum_{\bfi \in \calT^d_n} \left( \bfA_{\bfi} \cdot \prod_{j \in \{i_1,\dots, i_d \}} x_j\right),
$$
where $\{\bfA_{\bfi}\}_{\bfi \in \calT^d_n}$ is a multi-indexed sequence of deterministic matrices of the same dimension. $\bfF(\bfx)$ is permutation symmetric if $\bfA_{i_1, \dots, i_d} = \bfA_{i_{\sigma(1)}, \dots, i_{\sigma(d)}}$ for any permutation $\sigma \in \mathfrak{S}_d$ and any $(i_1, \dots, i_d) \in \calT^d_n$.
\end{definition}
\begin{remark}
The polynomial random matrix (\ref{not symmetric}) in \cref{example} is permutation symmetric if we formally rewrite it as
\begin{align*}
    \sum_{\bfi \in \calT^2_3} \bfB_\bfi \cdot x_ix_j = &  x_1x_2\begin{bmatrix} \frac{1}{2} & 0 \\ 0 & 0 \end{bmatrix} +  x_2x_1\begin{bmatrix} \frac{1}{2} & 0 \\ 0 & 0 \end{bmatrix} + x_1x_3\begin{bmatrix} 0 & 0 \\ 0 & \frac{1}{2} \end{bmatrix} + x_3x_1\begin{bmatrix} 0 & 0 \\ 0 & \frac{1}{2} \end{bmatrix} + x_2x_3\begin{bmatrix} 0 & \frac{1}{2} \\ \frac{1}{2} & 0 \end{bmatrix} + 
    x_3x_2\begin{bmatrix} 0 & \frac{1}{2} \\ \frac{1}{2} & 0 \end{bmatrix}. 
\end{align*}
The purpose of this technical detail will become clear in \cref{quadratic} and \cref{homo multilinear recursion}.
\end{remark}

\vspace{-10 pt}

\paragraph{Partial derivatives.} \label{partial} Let $\bfF(\bfx)$ be an $n$-variate polynomial random matrix of degree $D$. For $a,b,c \in \Z_{\geq 0}$ and $d = a + b + c$, we consider the block matrices $\bfF_{a,b,c}$ containing all $d$-th order partial derivatives of $ \bfF(\bfx)$. We define $\bfF_{a,b,c}$ recursively as $\bfF_{0,0,0} = \bfF(\bfx)$ and
$$
 \bfF_{a+1,b,c} = \trans{\begin{bmatrix}
  \partial_{x_1} \bfF_{a,b,c} & \dots & \partial_{x_n} \bfF_{a,b,c}
\end{bmatrix}},
$$
$$
 \bfF_{a,b+1,c} = \begin{bmatrix}
  \partial_{x_1} \bfF_{a,b,c} & \dots & \partial_{x_n} \bfF_{a,b,c}
\end{bmatrix},
$$
and
$$
 \bfF_{a,b,c+1} = \begin{bmatrix}
  \partial_{x_1} \bfF_{a,b,c} & & \\ & \ddots & \\ & & \partial_{x_n} \bfF_{a,b,c}
\end{bmatrix}.
$$
It is evident from the definition that the block matrices $ \bfF_{a+1,b,c} $, $\bfF_{a,b+1,c} $, $\bfF_{a,b,c+1}$ are assembled from sub-blocks $\{ \partial_{x_i} \bfF_{a,b,c} \}_{i=1}^n$ arranged vertically, horizontally and  diagonally respectively. The order of increment between $a$ and $b$ doesn’t affect the resulting matrix, \ie 
$$\bfF_{a+1,b+1,c} = \begin{bmatrix}
  \partial_{x_1} \bfF_{a,b+1,c} \\ \vdots \\ \partial_{x_n} \bfF_{a,b+1,c}
\end{bmatrix} = \begin{bmatrix}
  \partial_{x_1} \bfF_{a+1,b,c} & \dots & \partial_{x_n} \bfF_{a+1,b,c}
\end{bmatrix} = \begin{bmatrix}
  \partial_{x_1}\partial_{x_1} \bfF_{a,b,c} & \dots & \partial_{x_1}\partial_{x_n} \bfF_{a,b,c} \\ & \vdots & \\  \partial_{x_n}\partial_{x_1} \bfF_{a,b,c} & \dots & \partial_{x_n}\partial_{x_n} \bfF_{a,b,c}
\end{bmatrix}. $$
However, the order of increment involving $c$ affects the resulting matrix, \ie 
$$
\begin{bmatrix}
  \partial_{x_1} \bfF_{a,b,c+1} & \dots & \partial_{x_n} \bfF_{a,b,c+1}
\end{bmatrix}  \neq \begin{bmatrix}
  \partial_{x_1} \bfF_{a,b+1,c} & & \\ & \ddots & \\ & & \partial_{x_n} \bfF_{a,b+1,c}
\end{bmatrix}.
$$
Nevertheless, the order of increment doesn’t affect the Schatten norm. Since we will only be interested in the Schatten norms of these matrices, we will simply label them as $\bfF_{a,b+1,c+1}$ without specifying the order. It is easy to see that $\bfF_{a,b,c}$ is a deterministic matrix if $a+b+c = D$  and $\bfF_{a,b,c} = {\bf 0}$ if $a+b+c > D$.

\vspace{-10 pt}

\paragraph{Non-Hermitian matrices.} 
Many results for Hermitian matrices can be easily extended to non-Hermitian matrices by the following argument.

\begin{definition} [Hermitian Dilation, see \cite{Tro15} Sec. 2.1.16]
\label{dilation}
For each $\bfA \in \R^{d_1 \times d_2}$, the Hermitian dilation $\mathcal{H}: \R^{d_1 \times d_2} \rightarrow \Hsymb^{(d_1 + d_2) \times (d_1 + d_2)}$ is defined as
$$
\calH (\bfA) = \begin{bmatrix} \bf 0 & \bfA \\ \trans{\bfA} & \bf 0 \end{bmatrix}.
$$
\end{definition}
Since the square of the Hermitian dilation satisfies $\calH (\bfA)^2 = \begin{bmatrix} \bfA \trans{\bfA} & \bf 0 \\ \bf 0 & \trans{\bfA}\bfA \end{bmatrix} $, it follows that $\norm{\calH(\bfA)} = \norm{\bfA}$. 

\vspace{-10 pt}

\paragraph{Constants.} 
We denote $C$ for some universal constants and $C_a$ for constants depending only on some parameter $a$. The values of the constants may differ from one instance to another.

\section{Decoupling Inequalities}\label{sec:decoupling}
In this work, we focus on decoupling inequalities for moments. The following decoupling inequality holds in surprising generality. 

\begin{lemma}[Decoupling Inequality, see \cite{PG:book} Theorem 3.1.1] \label{generic decoupling}
For natural numbers $n \geq m$, let $\bfx = \{x_i\}^n_{i=1}$ be a sequence of $n$ independent random variables in a measurable space $(S, \cal S)$, and let $\{x_i^{(k)}\}^n_{i=1}$, $k = 1, \dots, m$, be $m$ independent copies of this sequence. Let $B$ be a separable Banach space, and let $h_{\bfi}: S^m \rightarrow B$ be a measurable function such that $\E (\norm{h_{\bfi} (x_{i_1}, \dots, x_{i_m})}) \leq \infty $ for each $\bfi \in \calT^m_n$. Let $\Phi: [0, \infty) \rightarrow [0,\infty)$ be a convex nondecreasing function such that $\E \Phi (\norm{h_{\bfi} (x_{i_1}, \dots, x_{i_m})}) \leq \infty $ for all $\bfi \in \calT^m_n$. Then, 
$$
\E \Phi \left( \norm{\sum_{\bfi \in \calT^m_n} h_{\bfi} (x_{i_1}, \dots, x_{i_m})} \right) \leq \E \Phi \left( C_m\norm{\sum_{\bfi \in \calT^m_n} h_{\bfi} (x_{i_1}^{(1)}, \dots, x_{i_m}^{(m)})} \right),
$$
where $C_m = 2^m \cdot (m^m-1) \cdot ((m-1)^{(m-1)} - 1) \cdot \ldots \cdot 3.$
\end{lemma} 

In general, the decoupling constant $C_m$ is super-exponential in $m$, which might seem quite large. But nevertheless $C_m$ is independent of the dimension $n$. Furthermore, $C_m$ can be improved significantly for some special classes of functions $h$ and certain underlying distributions of $\bfx$. For instance, \cite{KW15} showed that if $h$ is a homogeneous multilinear form of degree $m$, then $C_m = m^m$. Moreover, if we assume $\bfx$ is a Gaussian random vector, $C_m$ further improves to $m^{m/2}$. 

While the proof of decoupling inequality for tail probability is slightly involved \cite{PMS95}, the proof of decoupling inequalities for moments are fairly elementary and can be interpreted combinatorially as partitioning the random vector $\bfx$. We present the following lemma that generalizes Lemma 6.21 \cite{Rau10} to arbitrary degree $d$.

\begin{lemma} \label{decoupling}
    Let $\bfx = \{x_j\}_{j = 1}^n$ be a sequence of independent random variables with $\E x_j = 0$ for all $j = 1 ,\dots, n$. Let $\{\bfB_{\bfi}\}_{\bfi\in \calT_n^d}$ be a multi-indexed sequence of deterministic matrices of the same dimension. Then for $1\leq t \leq \infty$
    $$
    \E \norm{\sum_{\bfi \in \calT_n^d} \bfB_{\bfi} \cdot x_{i_1} \dots x_{i_d}}_{2t} \leq d^d \cdot \E \norm{\sum_{\bfi \in \calT_n^d} \bfB_{\bfi} \cdot x_{i_1}^{(1)} \dots x_{i_d}^{(d)}}_{2t}, $$
    where  $\bfx^{(1)}, \dots,  \bfx^{(d)}$ denote $d$ independent copies of $\bfx$.
\end{lemma}

\begin{proof}
    Consider a random partition of $\bf x$ into $d$ parts and let $\bf r$ be the \textit{partitioner}. Formally, ${\bf r} = (r_1, \dots, r_n)$ is a sequence of independent random variables with each $r_k$ uniformly distributed on $\{1, \dots, d\}$, \ie for $1\leq k \leq n$,
    $$\Psymb (r_k = 1) = \Psymb (r_k = 2) = \dots = \Psymb (r_k = d) = 1/d. $$
    Define an event $E:= \{r_{i_1} = 1, \cdots, r_{i_d} = d\}$ for each $\bfi \in \calT_n^d$. Conditioned on $\bfx$, 
    $$\E_{\bf r} \left[ \mathbb{1}_E (r_{i_1}, \dots, r_{i_d}) \right] = 1/d^d.$$ 
It follows that
\begin{align*}
        F:&= ~\E_{\bfx} \norm{\sum_{\bfi \in \calT_n^d} \bfB_{\bfi} \cdot x_{i_1} \dots x_{i_d}}_{2t} \\
        & = ~d^d \cdot \E_{\bfx} \norm{\E_{\bf r} \left [ \sum_{\bfi \in \calT_n^d} \mathbb{1}_{E} (r_{i_1}, \dots, r_{i_d}) \cdot \bfB_{\bfi} \cdot x_{i_1} \dots x_{i_d} \right] }_{2t} \\
        &\leq ~d^d \cdot \E_{\bf r} \E_\bfx \norm{\sum_{\bfi \in \calT_n^d} \mathbb{1}_{E} (r_{i_1}, \dots, r_{i_d}) \cdot \bfB_{\bfi} \cdot x_{i_1} \dots x_{i_d}}_{2t}, 
\end{align*}
where the last step is due to Jensen’s inequality and Fubini’s theorem.
Since the expectation over $\bf r$ satisfies the inequality, this implies the existence of an ${\bf r}^*$ satisfying the inequality. 

Fix ${\bf r}^*$ and define the \textit{partitioning} with respect to ${\bf r}^*$ as
$$P_1 = \{k\in [n]: r^*_k = 1 \}, ~\dots ~, P_d = \{k \in [n]: r^*_k = d\}.$$
Since these partitions have no intersections and $x_j$’s are independent random variables, replacing $x_{i_1}\cdots x_{i_d}$ with $x_{i_1}^{(1)}\cdots x_{i_d}^{(d)}$ for any $\bfi \in P_1 \times \cdots \times P_d$ will not change the distribution. Hence
\begin{equation} \label{eq1}
    F ~\leq ~ d^d \cdot \E \norm{\sum_{\bfi \in P_1 \times \cdots \times P_d} \bfB_{\bfi} \cdot x_{i_1}^{(1)} \dots x_{i_d}^{(d)}}_{2t}. 
\end{equation}

Denote all variables in $P_1 \times \cdots \times P_d$ by
$$
\calX:= \{x_{i_1}^{(1)}, ~\ldots, ~ x_{i_d}^{(d)}:  ~\forall~ i_1\in P_1, ~\ldots~, i_d \in P_d\},
$$
and denote the rest of the variables by $\calX^c$. It remains to show that the sum on the right-hand side of (\ref{eq1}) is over all ${\bf i} \in \calT^d_n$.
Observe that
\begin{equation} \label{eq2}
     \E_{\calX^c} \left[ \sum_{\bfi \not \in P_1 \times \cdots \times P_d} \bfB_{\bfi} \cdot x_{i_1}^{(1)} \dots x_{i_d}^{(d)} ~ ~\Biggr \rvert ~~ \calX ~~ \right] = {\bf 0}.
\end{equation}
Since the sum in (\ref{eq2}) has conditional expectation zero, we can add it to (\ref{eq1}) to get
$$
F \leq d^d \cdot \E_{\calX} \norm{\sum_{\bfi \in P_1 \times \cdots \times P_d} \bfB_{\bfi} \cdot x_{i_1}^{(1)} \dots x_{i_d}^{(d)} ~ + ~ \E_{\calX^c} \left[ \sum_{\bfi \not \in P_1 \times \cdots \times P_d} \bfB_{\bfi} \cdot x_{i_1}^{(1)} \dots x_{i_d}^{(d)} ~ ~\Biggr \rvert ~~ \calX ~~ \right]}_{2t}.
$$
Since the two sums on the right-hand side are independent conditioned on $\calX$, 
$$
F \leq d^d \cdot \E_{\calX} \norm{ \E_{\calX^c} \left[ \sum_{\bfi \in P_1 \times \cdots \times P_d} \bfB_{\bfi} \cdot x_{i_1}^{(1)} \dots x_{i_d}^{(d)} ~ + \sum_{\bfi \not \in P_1 \times \cdots \times P_d} \bfB_{\bfi} \cdot x_{i_1}^{(1)} \dots x_{i_d}^{(d)} ~ ~\Biggr \rvert ~~ \calX ~~ \right]}_{2t}.
$$
A conditional application of Jensen’s inequality yields the desired inequality,
\[
F \leq d^d \cdot \E_{{\bf x}^{(1)} \dots {\bf x}^{(d)}} \norm{\sum_{\bfi \in \calT^d_n} \bfB_{\bfi} \cdot x_{i_1}^{(1)} \dots x_{i_d}^{(d)}}_{2t}. 
\eqno \qedhere
\]
\end{proof}

\subsection{Improved Decoupling for Random Variables with Structured Indices}\label{sec:structured}
For a generic degree $d$ multilinear polynomial random matrices, an application of \cref{decoupling} yields a decoupling constant of $d^d$, which might not be optimal for polynomial random matrices with additional structures. We will show that the additional structures enable us to give a tighter decoupling inequality through a careful partitioning.

Building on \cref{decoupling}, we prove a decoupling inequality for polynomial random matrices, in which the indices of random variables have a graph structure. Let $G=(V,E)$ be a fixed simple graph with $|V|=k$, and $|E|=d$. We denote the vertices by $V(G) = \{v_1, \ldots, v_k\}$ and the ordered edge set by $E(G) = \{(i_1, j_1), \ldots, (i_d, j_d)\}$. For any $\phi \in \calT^k_n$, we write $\phi(i)$ for the $i$-th component of $\phi$ and $\phi(E)$ for $\{(\phi(i_1), \phi(j_1)), \ldots, (\phi(i_d), \phi(j_d))\}$.
\begin{lemma} \label{graph_decoupling}
    Let $\bfz = \{Z_{\phi(i), \phi(j)}\}_{\phi \in \calT^k_n, (i,j)\in E}$ be a double sequence of independent random variables with $\E Z_{\phi(i),\phi(j)} = 0$ for all $(i,j) \in E$ and $ \phi \in \calT^k_n$. Let $\{\bfB_{\phi(E)}\}_{\phi \in \calT^k_n}$ be a multi-indexed sequence of deterministic matrices of the same dimension. Then for $d\leq k(k-1)$ and $1 \leq t \leq \infty$,
    $$
    \E \norm{\sum_{\phi \in \calT_n^k} \bfB_{\phi(E)} \cdot \prod_{(i,j) \in E} Z_{\phi(i), \phi(j)}}_{2t} \leq k^k \cdot \E \norm{\sum_{\phi \in \calT_n^k}\bfB_{\phi(E)} \cdot Z_{\phi(i_1), \phi(j_1)}^{(1)}\cdots Z_{\phi(i_d), \phi(j_d)}^{(d)}}_{2t}, $$
    where  $\bfz^{(1)}, \ldots, \bfz^{(d)}$ denote $d$ independent copies of $\bfz$.
\end{lemma} 
\begin{remark}
   It is crucial for the indices of all monomials to share the same structure so that $d$ independent copies of $\bfz$ suffices. Otherwise, $k(k-1)$ independent copies are needed in general. 
\end{remark}

\begin{proof}
    Consider a random partition of the ground set $[n]$ into $k$ parts and let $\bf r$ be the \textit{vertex partitioner}. Formally, let ${\bf r} = (r_1, \dots, r_n)$ be a sequence of independent random variables with each $r_p$ uniformly distributed on $\{1, \dots, k\}$, \ie for $1\leq p \leq n$,
    $$\Psymb (r_p = 1) = \Psymb (r_p = 2) = \dots = \Psymb (r_p = k) = 1/k. $$
    Define an event $A: = \{r_{\phi(v_1)} = 1, \cdots, r_{\phi(v_k)} = k\}$ for each $\phi \in \calT^k_n$. Conditioned on $\bfz$, 
    $$\E_{\bf r} \left[ \mathbb{1}_{A} (r_{\phi(v_1)}, \dots, r_{\phi(v_k)}) \right] = 1/k^k.$$
    It follows that
\begin{align*}
        F:&= \E_{\bfz} \norm{\sum_{\phi \in \calT_n^k} \bfB_{\phi(E)} \cdot \prod_{(i,j) \in E} Z_{\phi(i), \phi(j)} }_{2t} \\
        & = k^k \cdot \E_{\bfz} \norm{\E_{\bf r} \left [ \sum_{\phi \in \calT_n^k}  \mathbb{1}_{A} (r_{\phi(v_1)}, \dots, r_{\phi(v_k)}) \cdot \bfB_{\phi(E)} \cdot \prod_{(i,j) \in E} Z_{\phi(i), \phi(j)} \right] }_{2t} \\
        &\leq k^k \cdot \E_{\bf r} \E_\bfz \norm{\sum_{\phi \in \calT_n^k} \mathbb{1}_{A} (r_{\phi(v_1)}, \dots, r_{\phi(v_k)}) \cdot \bfB_{\phi(E)} \cdot \prod_{(i,j) \in E} Z_{\phi(i), \phi(j)}}_{2t}, 
\end{align*}
where the last step is due to Jensen’s inequality and Fubini’s theorem.
Since the expectation over $\bf r$ satisfies the inequality, this implies the existence of an ${\bf r}^*$ satisfying the inequality. 

Fix ${\bf r}^*$ and define the \textit{vertex partitioning} with respect to ${\bf r}^*$ as
$$P_1 = \{p\in [n]: r^*_p = 1 \}, ~\ldots~ , P_k = \{p\in [n]: r^*_p = k\},$$ 
which in turn induces an \textit{edge partitioning}, 
$$ P_e ~ := ~\{(P_i, P_j)\}_{(i,j)\in \calT^2_k}. $$
Notice that there are $k(k-1)$ elements in $P_e$, but we only have $d$ edges to partition. So we select $d$ elements out of $P_e$ in the following way. Fix an arbitrary $\phi^* \in \calT^k_n$, choose 
$$
P^*_1 = (P_{\phi^*(i_1)}, P_{\phi^*(j_1)}), ~ \ldots~,  P^*_d = (P_{\phi^*(i_d)}, P_{\phi^*(j_d)}).
$$
Let’s call $\{P_1^*, \ldots P^*_d\}$ a \textit{d-edge partitioning}. In other words, we may select any $d$ elements out of $P_e$ to form a \textit{d-edge partitioning} so long as the pattern of the indices of $(P_i, P_j)$’s matches that of $E(G)$.
Since the elements in $P_e$ have no intersections, the partitions in $\{P_1^*, \ldots, P^*_d\}$ also have no intersections. Since $Z_{\phi(i),\phi(j)}$’s are independent random variables, replacing $Z_{\phi(i_1), \phi(j_1)} \cdots Z_{\phi(i_d), \phi(j_d)}$ with $Z_{\phi(i_1), \phi(j_1)}^{(1)} \cdots Z_{\phi(i_d), \phi(j_d)}^{(d)}$ for any $\phi(E) \in P^*_1 \times \cdots \times P^*_d$ will not change the distribution. Hence
\begin{equation} \label{eq3}
  F \leq k^k \cdot \E \norm{\sum_{\phi(E) \in P^*_1\times \cdots \times P^*_d} \bfB_{\phi(E)} \cdot Z_{\phi(i_1), \phi(j_1)}^{(1)} \cdots Z_{\phi(i_d), \phi(j_d)}^{(d)}}_{2t}.   
\end{equation}

Denote all variables in $P^*_1 \times \cdots \times P^*_d$ by
$$
\calZ:= \left \{Z_{\phi(i_1), \phi(j_1)}^{(1)}, ~\ldots, ~ Z_{\phi(i_d), \phi(j_d)}^{(d)} :  ~\forall~ (\phi(i_1), \phi(j_1)) \in P_1^*, ~\ldots~, (\phi(i_d), \phi(j_d)) \in P_d^*\right\},
$$
and denote the rest of the variables by $\calZ^c$.
It remains to show that the sum on the right-hand side of (\ref{eq3}) is over all $\phi \in \calT^k_n$.
Observe that
\begin{equation} \label{eq4}
  \E_{\calZ^c} \left[ \sum_{\phi(E) \not \in P^*_1\times \cdots \times P^*_d} \bfB_{\phi(E)} \cdot Z_{\phi(i_1), \phi(j_1)}^{(1)} \cdots Z_{\phi(i_d), \phi(j_d)}^{(d)} ~ ~\Biggr \rvert ~~ \calZ ~~\right] = {\bf 0}.  
\end{equation}
Denote $P^*_1\times \cdots \times P^*_d$ by $\calP$. Since the sum in (\ref{eq4}) has conditional expectation zero, we can add it to (\ref{eq3}) to get
$$
F \leq k^k \cdot \E_{\calZ} \norm{\sum_{\phi(E) \in \calP} \bfB_{\phi(E)} \cdot Z_{\phi(i_1), \phi(j_1)}^{(1)} \cdots Z_{\phi(i_d), \phi(j_d)}^{(d)} + \E_{\calZ^c} \left[ \sum_{\phi(E) \not \in \calP} \bfB_{\phi(E)} \cdot Z_{\phi(i_1), \phi(j_1)}^{(1)} \cdots Z_{\phi(i_d), \phi(j_d)}^{(d)} ~ ~\Biggr \rvert ~~ \calZ ~~\right]}_{2t}.
$$
Since the two sums on the right-hand side are independent conditioned on $\calZ$, 
$$
F \leq k^k \cdot \E_{\calZ} \norm{\E_{\calZ^c} \left[ \sum_{\phi(E) \in \calP} \bfB_{\phi(E)} \cdot Z_{\phi(i_1), \phi(j_1)}^{(1)} \cdots Z_{\phi(i_d), \phi(j_d)}^{(d)} +  \sum_{\phi(E) \not \in \calP} \bfB_{\phi(E)} \cdot Z_{\phi(i_1), \phi(j_1)}^{(1)} \cdots Z_{\phi(i_d), \phi(j_d)}^{(d)} ~ ~\Biggr \rvert ~~ \calZ ~~\right]}_{2t}.
$$
A conditional application of Jensen’s inequality yields the desired inequality,
$$
F \leq k^k \cdot \E_{\bfz^{(1)}, \ldots, \bfz^{(d)}} \norm{\sum_{\phi \in \calT_n^k}  \bfB_{\phi(E)} \cdot Z_{\phi(i_1), \phi(j_1)}^{(1)} \cdots Z_{\phi(i_d), \phi(j_d)}^{(d)}}_{2t}. 
\eqno \qedhere
$$ 
\end{proof}

\section{Multilinear Polynomial Random Matrices}\label{sec:multilinear}
In this section, we prove a moment bound for multilinear polynomial random matrices with bounded, normalized random variables. We follow our recursion framework by decoupling the polynomial random matrices first and then apply the matrix Rosenthal inequality recursively to sums of (conditionally) independent random matrices. 

To this end, we derive the non-Hermitian matrix Rosenthal inequality from the Hermitian matrix Rosenthal inequality by Hermitian dilation.

\begin{lemma} [Non-Hermitian Matrix Rosenthal Inequality]
\label{rosenthal}
Suppose that $t=1$ or $t \geq 1.5$. Consider a finite sequence $\{\bfY_k\}_{k\geq 1}$ of centered, independent, random matrices, and assume that $\E \norm{\bfY_k}_{4t}^{4t} < \infty$. Then
$$
    \E \norm{\sum_{k}\bfY_k}_{4t}^{4t} \leq (16t)^{3t} \cdot \left\{ \norm{\left(\sum_k\E  \bfY_k \trans{\bfY}_k \right)^{1/2}}_{4t}^{4t} + \norm{ \left(\sum_k\E \trans{\bfY}_k\bfY_k \right)^{1/2} }_{4t}^{4t} \right\} + (8t)^{4t}\cdot \left(\sum_k \E \norm{\bfY_k}_{4t}^{4t} \right).
$$
\end{lemma}

\begin{proof} 
For a sequence of Hermitian matrices $\{\bfX_k\}_{k \geq 1}$ satisfying all the assumptions above, we have (Hermitian) matrix Rosenthal inequality \cite{MJC12} Corollary 7.4,
$$
\E \norm{\sum_{k}\bfX_k}_{4t}^{4t} \leq (16t)^{3t} \cdot \norm{\left(\sum_k\E \bfX^2_k \right)^{1/2}}_{4t}^{4t} + (8t)^{4t}\cdot \sum_k \E \norm{\bfX_k}_{4t}^{4t}.
$$
We use Hermitian dilation to extend the inequality to non-Hermitian matrices by setting $\bfX_k = \calH(\bfY_k)$ and notice that 
$$
\E \norm{\sum_{k} \begin{bmatrix}
{\bf 0} & \bfY_k \\ \trans{\bfY}_k & {\bf 0}
\end{bmatrix}}_{4t}^{4t} = \E \tr \left(\begin{bmatrix}
((\sum_{k} \bfY_k)(\sum_{k} \trans{\bfY}_k))^{2t} & {\bf 0} \\  {\bf 0} & ((\sum_{k} \trans{\bfY}_k)(\sum_{k} \bfY_k))^{2t} 
\end{bmatrix}\right) = 2\E \norm{\sum_k \bfY_k}_{4t}^{4t},
$$
\begin{align*}
      \norm{\left(\sum_k\E \begin{bmatrix} \bfY_k \trans{\bfY}_k & {\bf 0} \\ {\bf 0} & \trans{\bfY}_k \bfY_k \end{bmatrix} \right)^{1/2}}_{4t}^{4t} &= \tr \left(\begin{bmatrix}\sum_k\E  \bfY_k \trans{\bfY}_k & \bf 0 \\ \bf 0 & \sum_k\E \trans{\bfY}_k\bfY_k \end{bmatrix} \right)^{2t}  \\
      &= \norm{\left(\sum_k\E  \bfY_k \trans{\bfY}_k \right)^{1/2}}_{4t}^{4t} + \norm{ \left(\sum_k\E \trans{\bfY}_k \bfY_k \right)^{1/2} }_{4t}^{4t}, \\
      \sum_k \E \norm{\begin{bmatrix} {\bf 0} & \bfY_k \\ \trans{\bfY}_k & {\bf 0} \end{bmatrix}}_{4t}^{4t}  &= 2 \sum_k \E \norm{\bfY_k}_{4t}^{4t}. 
\end{align*}
Hence, the result follows.
\end{proof}

For the ease of notation, we demonstrate our recursion framework in a special case of quadratic form first. The multilinear form follows exactly the same reasoning.

\begin{theorem} [Quadratic Recursion]
\label{quadratic}
Let $\{x_i\}_{i = 1}^n$ be a sequence of i.i.d random variables with $\E x_i= 0$,  $\E x_i^2= 1$ and $|x_i| \leq L$ for all $1\leq i \leq n$. Let $\{\bfA_{i,j}\}_{(i,j) \in \calT^2_n}$ be a double sequence of deterministic matrices of the same dimension. Define 
$$
\bfF(\bfx) = \sum_{(i,j) \in \calT^2_n} \bfA_{i,j} x_{i}x_{j}.
$$
Suppose that $\bfF(\bfx)$ is also permutation symmetric, \ie $\bfA_{i,j} = \bfA_{j,i}$, $\forall (i,j) \in \calT^2_n$. Then for $a,b,c\in \Z_{\geq 0}$ and $2 \leq t < \infty$,
$$
\E \norm{\bfF(\bfx) - \E \bfF(\bfx)}_{4t} ~ \leq ~ 2 (32t)^{2} \cdot \left( \sum_{\substack{a,b,c: \\a+b+c=2 }} L^c \cdot \norm{ \bfF_{a,b,c}}_{4t} \right).
$$
\end{theorem}
\begin{remark}
For $a,b,c\in \Z_{\geq 0}$ such that $a+b+c = 2$, the partial derivative matrices $\{\bfF_{a,b,c}\}$ are block matrices whose blocks consist of $\{\bfA_{i,j}\}_{(i,j) \in \calT^2_n}$ and zero matrices. For instance,
$$
\bfF_{1,1,0} = \begin{bmatrix}
    \bf 0 & \bfA_{1,2} & \dots & \bfA_{1,n} \\
    \bfA_{2,1} & \bf 0 & \dots & \bfA_{2,n} \\
    \vdots & \vdots & \ddots & \vdots \\
     \bfA_{n,1} & \bfA_{n,2} & \dots & \bf 0 
\end{bmatrix},
$$
$$
\bfF_{0,2,0} = \begin{bmatrix}
 \bf 0 & \bfA_{1,2} & \dots & \bfA_{1,n} \rvert & \dots & \rvert \bfA_{n,1} & \dots & \bfA_{n,n-1} & \bf 0
\end{bmatrix},
$$
and 
$$
\bfF_{0,1,1} = \begin{bmatrix}
    \bf 0 & \bfA_{1,2} & \dots & \bfA_{1,n} \\
     & & & & \bfA_{2,1} & \bf 0 & \dots & \bfA_{2,n} \\
     & & & & & & & & \ddots \\
     & & & & & & & & & \bfA_{n,1}  & \bfA_{n,2} & \dots & \bf 0 \\
\end{bmatrix}.
$$
\end{remark}
\begin{proof}
Let $\bfx’ = \{x’_j\}_{j = 1}^n$ be an independent copy of $\bfx = \{x_i\}_{i = 1}^n$. We decouple $\bfF(\bfx)$ by \cref{decoupling},
$$
 E:=\E \norm{\sum_{(i,j) \in \calT^2_n} \bfA_{i,j} x_{i}x_{j}}_{4t}^{4t} \leq 4^{4t} \cdot \E \norm{\sum_{(i,j) \in \calT^2_n} \bfA_{i,j} x_{i} x’_{j}}_{4t}^{4t} .
$$
Since the decoupled $\bfF(\bfx)$ on the right hand side is a function of $\bfx$ and $\bfx’$, the decoupled $\bfF(\bfx)$ can be viewed as taking two arguments instead. Thus we will refer to the decoupled $\bfF(\bfx)$ as $\bfF(\bfx, \bfx’)$. 

By the law of total expectation,
$$
E ~\leq ~ 4^{4t} \cdot \E \left( \E_{\bfx} \left( \norm{\sum_{i=1}^n x_{i}\left( \sum_{j=1}^n  \bfA_{i,j} x_{j}’ \right)}_{4t}^{4t} ~ \Biggr \rvert ~ \bfx’ \right) \right).
$$
Since $\bfx $ and $\bfx’$ are independent, if we fix $\bfx’$, the summation becomes a sum of centered, (conditionally) independent random matrices in $\bfx$. So we can apply matrix Rosenthal inequality (\cref{rosenthal}) and use the fact that $\E x_i^2= 1$, and $x_i$’s are independent to get
\begin{align} 
    E~ &  \leq~  4^{4t} (16t)^{3t} \cdot \E \norm{ \left(\sum_{i=1}^n \left( \sum_{j=1}^n  \bfA_{i,j} x’_{j} \right) \trans{\left( \sum_{j=1}^n  \bfA_{i,j} x’_{j} \right)}\right)^{1/2}}_{4t}^{4t} \label{E1} \\ & +  4^{4t} (16t)^{3t}\cdot \E \norm {\left( \sum_{i=1}^n \trans{\left( \sum_{j=1}^n  \bfA_{i,j} x’_{j} \right)} \left( \sum_{j=1}^n  \bfA_{i,j} x’_{j} \right) \right)^{1/2} }_{4t}^{4t}  \label{E2} \\
    &+  4^{4t} (8t)^{4t} L^{4t} \cdot \sum_{i=1}^n  \E \norm { \sum_{j=1}^n  \bfA_{i,j} x’_{j}}_{4t}^{4t}. \label{E3}
\end{align}
Recall that $\bfF(\bfx, \bfx’)$ refers to the decoupled $\bfF(\bfx)$. The expression in (\ref{E1}) satisfies
\begin{align} \label{obs}
    \sum_{i=1}^n \left( \sum_{j=1}^n  \bfA_{i,j} x’_{j} \right) \trans{\left( \sum_{j=1}^n  \bfA_{i,j} x’_{j} \right)} & = \begin{bmatrix}
        \partial_{x_1} \bfF(\bfx, \bfx’) & \dots & \partial_{x_n} \bfF(\bfx, \bfx’)
    \end{bmatrix}\begin{bmatrix}
        \partial_{x_1} \trans{\bfF}(\bfx, \bfx’) \\ \vdots \\ \partial_{x_n} \trans{\bfF}(\bfx, \bfx’) 
    \end{bmatrix}.
\end{align}
Notice that the two sides are merely equal to each other in a formal way. We didn’t actually take the partial derivatives since $x_i$’s could be discrete random variables. But conceptually, it is very convenient to think of this step as taking partial derivatives in $\bfx$. 

Following (\ref{obs}) and using our notation for partial derivative block matrices introduced in \cref{partial}, we can write (\ref{E1}) compactly as
\begin{align} 
   & \E \norm{\left(\sum_{i=1}^n \left( \sum_{j=1}^n  \bfA_{i,j} x’_{j} \right) \trans{\left( \sum_{j=1}^n  \bfA_{i,j} x’_{j} \right)}\right)^{1/2}}^{4t}_{4t} \notag \\ = & \E \norm{\left(\begin{bmatrix}
        \partial_{x_1} \bfF(\bfx, \bfx’) & \dots & \partial_{x_n} \bfF(\bfx, \bfx’)
    \end{bmatrix}\begin{bmatrix}
        \partial_{x_1} \trans{\bfF}(\bfx, \bfx’) \\ \vdots \\ \partial_{x_n} \trans{\bfF}(\bfx, \bfx’) 
    \end{bmatrix}\right)^{1/2}}^{4t}_{4t} = \E \norm{\bfF_{0,1,0}(\bfx’)}^{4t}_{4t}. \label{obs1}
\end{align}
Similarly, (\ref{E2}) and (\ref{E3}) can be written compactly as
\begin{align}  
    & \E \norm{\left(\sum_{i=1}^n \trans{\left( \sum_{j=1}^n  \bfA_{i,j} x’_{j} \right)} \left( \sum_{j=1}^n  \bfA_{i,j} x’_{j} \right)\right)^{1/2}}_{4t}^{4t}  \notag \\ = & \E \norm{\left(\begin{bmatrix}
        \partial_{x_1} \trans{\bfF}(\bfx, \bfx’) & \dots & \partial_{x_n} \trans{\bfF}(\bfx, \bfx’)
    \end{bmatrix}\begin{bmatrix}
        \partial_{x_1} \bfF(\bfx, \bfx’) \\ \vdots \\ \partial_{x_n} \bfF(\bfx, \bfx’) 
    \end{bmatrix}\right)^{1/2}}_{4t}^{4t} = \E \norm{\bfF_{1,0,0}(\bfx’)}^{4t}_{4t}, \label{obs2}
\end{align}
and
\begin{align}
&\sum_{i=1}^n \E \norm{\sum_{j=1}^n  \bfA_{i,j} x’_{j}}_{4t}^{4t} \notag \\ = &\E \tr \begin{bmatrix}
\partial_{x_1} \bfF(\bfx, \bfx’) \cdot \partial_{x_1} \trans{\bfF}(\bfx, \bfx’) & & \\ & \ddots & \\ & & \partial_{x_n} \bfF(\bfx, \bfx’) \cdot \partial_{x_n} \trans{\bfF}(\bfx, \bfx’) 
\end{bmatrix}^{2t} =  \E \norm{\bfF_{0,0,1}(\bfx’)}^{4t}_{4t}.  \label{obs3} 
\end{align}

Hence the inequality (\ref{E1}-\ref{E3}) can be written succinctly as
$$
E \leq  4^{4t} (16t)^{3t} \cdot \left( \E \norm{\bfF_{1,0,0}}_{4t}^{4t} + \E \norm{\bfF_{0,1,0}}_{4t}^{4t}  \right) +  4^{4t} (8t)^{4t}L^{4t} \cdot \E \norm{\bfF_{0,0,1}}_{4t}^{4t}.
$$

Since $\bfF(\bfx, \bfx’)$ is multilinear and we have already taken the partial derivatives in $\bfx$, $\bfF_{1,0,0}$, $\bfF_{0,1,0}$, and $\bfF_{0,0,1}$ are sums of centered, independent random matrices in $\bfx’$; \ie
$$\bfF_{1,0,0} (\bfx’) = \sum_{j = 1}^n x_j’\trans{\begin{bmatrix}
\trans{\bfA}_{1,j} & \dots & \trans{\bfA}_{n,j}
\end{bmatrix}},$$
$$\bfF_{0,1,0} (\bfx’) = \sum_{j = 1}^n x_j’\begin{bmatrix}
\bfA_{1,j} & \dots & \bfA_{n,j}
\end{bmatrix},$$ and 
$$\bfF_{0,0,1} (\bfx’) = \sum_{j = 1}^n x_j’\begin{bmatrix}
\bfA_{1,j} \\ & \ddots \\ && \bfA_{n,j}
\end{bmatrix}.$$ 
Therefore, we can apply matrix Rosenthal inequality (\cref{rosenthal}) again to bound $\E \norm{\bfF_{1,0,0}}^{4t}_{4t}$, $\E \norm{\bfF_{0,1,0}}^{4t}_{4t}$, and $\E \norm{\bfF_{0,0,1}}^{4t}_{4t}$. Observe that
$$
\begin{bmatrix}
   \partial_{x’_1} \bfF_{0,1,0} \\ \vdots \\ \partial_{x’_n} \bfF_{0,1,0}  
\end{bmatrix} = \begin{bmatrix}
    \bf 0 & \bfA_{2,1} & \dots & \bfA_{n,1} \\
    \bfA_{1,2} & \bf 0 & \dots & \bfA_{n,2} \\
    \vdots & \vdots & \ddots & \vdots \\
     \bfA_{1,n} & \bfA_{2,n} & \dots & \bf 0 
\end{bmatrix} =
\begin{bmatrix}
    \bf 0 & \bfA_{1,2} & \dots & \bfA_{1,n} \\
    \bfA_{2,1} & \bf 0 & \dots & \bfA_{2,n} \\
    \vdots & \vdots & \ddots & \vdots \\
     \bfA_{n,1} & \bfA_{n,2} & \dots & \bf 0 
\end{bmatrix} = \trans{\begin{bmatrix}
   \partial_{x’_1} \trans{\bfF}_{1,0,0} \\ \vdots \\ \partial_{x’_n} \trans{\bfF}_{1,0,0}  
\end{bmatrix}} = \bfF_{1,1,0},
$$
where the second equality is due to the permutation symmetric property of $\bfF(\bfx)$, \ie $\bfA_{i,j} = \bfA_{j,i}$, $\forall (i,j) \in \calT^2_n$. After applying two rounds of matrix Rosenthal inequality (\cref{rosenthal}), we have
$$
E < 2 \cdot 4^{4t} (16t)^{8t} \cdot \left(\norm{\bfF_{2,0,0}}_{4t}^{4t} + \norm{\bfF_{0,2,0}}_{4t}^{4t} + \norm{\bfF_{1,1,0}}_{4t}^{4t} +  L^{4t}\norm{\bfF_{1,0,1}}_{4t}^{4t} +  L^{4t}\norm{\bfF_{0,1,1}}_{4t}^{4t} +  L^{8t}\norm{\bfF_{0,0,2}}_{4t}^{4t}\right).
$$
We complete the proof by observing that $\E \bfF(\bfx) = {\bf 0}$ due to multilinearity of $\bfF(\bfx)$. 
\end{proof}

In general, differentiating the decoupled $\bfF(\bfx, \bfx’)$ with respect to $\bfx$ and $\bfx’$ can be quite different from taking second order derivatives of $\bfF(\bfx)$ with respect to $\bfx$. For example, it is clear that,
$$
\partial_{x_1}\partial_{x_2} \begin{bmatrix}
    0 & x_1x_2 \\ 0 & 0
\end{bmatrix} = \begin{bmatrix}
    0 & 1 \\ 0 & 0
\end{bmatrix} = \partial_{x_2}\partial_{x_1} \begin{bmatrix}
    0 & x_1x_2 \\ 0 & 0
\end{bmatrix},
$$
since $\partial x_1$ and $\partial x_2$ commute. But taking partial derivatives after decoupling gives,
$$
\partial_{x_1}\partial_{x’_2} \begin{bmatrix}
    0 & x_1x_2’ \\ 0 & 0
\end{bmatrix} = \begin{bmatrix}
    0 & 1 \\ 0 & 0
\end{bmatrix} \neq \begin{bmatrix}
    0 & 0 \\ 0 & 0
\end{bmatrix} = \partial_{x_2}\partial_{x’_1} \begin{bmatrix}
    0 & x_1x_2’ \\ 0 & 0
\end{bmatrix},
$$
where $x_1$, $x_2$, $x’_1$, $x’_2$ are actually four distinct variables. So one can’t expect the result of taking partial derivatives under two scenarios are the same. We could directly work with this by recording $\partial_{x_1}\partial_{x’_2}\bfF(\bfx)$ and $\partial_{x_2}\partial_{x’_1}\bfF(\bfx)$ separately, but the notation gets cumbersome very quickly. We would prefer to state our results in regular partial derivatives, which is more intuitive. So instead, we write $\bfF(\bfx)$ in the permutation symmetric way. Continuing with our example, we write
$$
\begin{bmatrix}
    0 & x_1x_2 \\ 0 & 0
\end{bmatrix} = \frac{1}{2}\begin{bmatrix}
    0 & x_1x_2 \\ 0 & 0
\end{bmatrix} + \frac{1}{2} \begin{bmatrix}
    0 & x_2x_1 \\ 0 & 0
\end{bmatrix}.
$$
This doesn’t affect the second order derivatives, but notice now what we get taking partial derivatives after decoupling,
$$
\partial_{x_1}\partial_{x’_2} \left( \frac{1}{2}\begin{bmatrix}
    0 & x_1x’_2 \\ 0 & 0
\end{bmatrix} + \frac{1}{2} \begin{bmatrix}
    0 & x_2x’_1 \\ 0 & 0
\end{bmatrix} \right) = \frac{1}{2} \begin{bmatrix}
    0 & 1 \\ 0 & 0
\end{bmatrix} = \partial_{x_2}\partial_{x’_1} \left( \frac{1}{2}\begin{bmatrix}
    0 & x_1x’_2 \\ 0 & 0
\end{bmatrix} + \frac{1}{2} \begin{bmatrix}
    0 & x_2x’_1 \\ 0 & 0
\end{bmatrix} \right).
$$
Therefore, the permutation symmetric property of $\bfF(\bfx)$ makes the resulting partial derivative matrices $\{\bfF_{a,b,c}\}$ under two scenarios differ only by a constant factor. This also applies to $\bfF(\bfx)$ of higher degree.

\begin{claim} \label{recurse rosenthal}
    Let $\bfF(\bfx)$ be a homogeneous multilinear polynomial random matrix of degree $D$. Let $\bfF(\bfx^{(1)}, \dots, \bfx^{(D)})$ be the decoupled $\bfF(\bfx)$, \ie
    $$
    \bfF(\bfx^{(1)}, \dots, \bfx^{(D)}) = \sum_{{\bf j} \in \calT^D_n} \bfA_{\bf j} \cdot x_{j_1}^{(1)} \cdots x_{j_D}^{(D)},
    $$ 
    where $\{\bfA_{\bf k}\}_{{\bf k}\in \calT^D_n}$ is a multi-indexed sequence of deterministic matrices of the same dimension.
    For some fixed $a,b,c\in \Z_{\geq 0}$ and $k = a+b+c <D$, let $\bfF_{a,b,c}$ be the block matrix of $k$-th order partial derivatives of $\bfF(\bfx^{(1)}, \dots, \bfx^{(D)})$. Then $\bfF_{a,b,c}$ is a homogeneous multilinear polynomial random matrix of degree $d = D-k$. Furthermore,
    $$
    \E \norm{\bfF_{a,b,c}}_{4t}^{4t} \leq (16t)^{3t} \cdot \left( \E \norm{{\bfF_{a+1, b, c}}}_{4t}^{4t} + \E \norm{\bfF_{a,b+1,c}}_{4t}^{4t}\right) + (8t)^{4t} \cdot \E \norm{\bfF_{a,b,c+1}}_{4t}^{4t}.
    $$    
\end{claim}
\begin{proof} Without loss of generality, we differentiate $\bfF(\bfx^{(1)}, \dots, \bfx^{(D)})$ with respect to $\bfx^{(D)}$ first, then $\bfx^{(D-1)}$, and so on. It is straightforward to see that after $k = a+b+c$ rounds of differentiation, $\bfF_{a,b,c}(\bfx^{(1)}, \dots, \bfx^{(d)})$ is a homogeneous polynomial random matrix of degree $d = D-k$. And
$$
\bfF_{a,b,c}(\bfx^{(1)}, \dots, \bfx^{(d)}) = \sum_{\bfi \in \calT^d_n} \bfB_{\bfi} \cdot x_{i_1}^{(1)} \cdots x_{i_d}^{(d)}, 
$$
where $\{\bfB_{\bfi}\}_{\bfi \in \calT^d_n}$ is a multi-indexed sequence of deterministic matrices of the same dimension. More specifically, each $\bfB_{\bfi}$ is a block matrix whose blocks consist of $\{ \bfA_{{\bf j}}\}_{{\bf j} \in \calT^D_n}$ such that $j_1 = i_1, \dots, j_d = i_d$. 

Conditioned on $\bfx^{(1)}, \dots, \bfx^{(d-1)}$, $ \bfF_{a,b,c}(\bfx^{(1)}, \dots, \bfx^{(d)})$ is a sum of centered, (conditionally) independent random matrices in $\bfx^{(d)}$. It follows that
\begin{align*}
& \E \left( \norm{\bfF_{a,b,c}(\bfx^{(1)}, \dots, \bfx^{(d)})}^{4t}_{4t} ~ \Biggr \rvert ~ \bfx^{(1)}, \dots \bfx^{(d-1)}\right) \\
= & ~ \E \left(\E_{\bfx^{(d)}} \left( \norm{\sum_{\bfi \in \calT^d_n} \bfB_{\bfi} \cdot x_{i_1}^{(1)} \cdots x_{i_d}^{(d)}}^{4t}_{4t} ~ \Biggr \rvert ~ \bfx^{(1)}, \dots \bfx^{(d-1)} \right) \right) \\
\leq & ~ (16t)^{3t} \cdot \E \norm{ \left(\sum_{i_d=1}^n \left(\sum_{\bfi \in \calT^d_n} \bfB_{\bfi} \cdot x_{i_1}^{(1)} \cdots x_{i_{d-1}}^{(d-1)}\right) \trans{\left( \sum_{\bfi \in \calT^d_n} \bfB_{\bfi} \cdot x_{i_1}^{(1)} \cdots x_{i_{d-1}}^{(d-1)} \right)}\right)^{1/2}}_{4t}^{4t} \\ 
& + (16t)^{3t}\cdot \E \norm {\left( \sum_{i_d =1}^n \trans{\left( \sum_{\bfi \in \calT^d_n}  \bfB_{\bfi} \cdot x_{i_1}^{(1)} \cdots x_{i_{d-1}}^{(d-1)} \right)} \left( \sum_{\bfi \in \calT^d_n} \bfB_{\bfi} \cdot x_{i_1}^{(1)} \cdots x_{i_{d-1}}^{(d-1)} \right) \right)^{1/2} }_{4t}^{4t}  \\
& + (8t)^{4t} L^{4t} \cdot \sum_{i_d=1}^n  \E \norm { \sum_{\bfi \in \calT^d_n}  \bfB_{\bfi} \cdot x_{i_1}^{(1)} \cdots x_{i_{d-1}}^{(d-1)}}_{4t}^{4t},
\end{align*}
where the inequality is due to matrix Rosenthal inequality (\cref{rosenthal}). Similar to observations (\ref{obs1}), (\ref{obs2}), and (\ref{obs3}) in the proof of \cref{quadratic}, we have
\begin{align*} 
   & \E \norm{ \left(\sum_{i_d=1}^n \left(\sum_{\bfi \in \calT^d_n} \bfB_{\bfi} \cdot x_{i_1}^{(1)} \cdots x_{i_{d-1}}^{(d-1)}\right) \trans{\left( \sum_{\bfi \in \calT^d_n} \bfB_{\bfi} \cdot x_{i_1}^{(1)} \cdots x_{i_{d-1}}^{(d-1)} \right)}\right)^{1/2}}_{4t}^{4t} \\ = & ~\E \norm{\left(\begin{bmatrix}
        \partial_{x_1^{(d)}} \bfF_{a,b,c}  & \dots & \partial_{x_n^{(d)}} \bfF_{a,b,c}
    \end{bmatrix}\begin{bmatrix}
        \partial_{x_1^{(d)}} \trans{\bfF}_{a,b,c} \\ \vdots \\ \partial_{x_n^{(d)}} \trans{\bfF}_{a,b,c}
    \end{bmatrix}\right)^{1/2}}^{4t}_{4t} = \E \norm{\bfF_{a,b+1,c}}^{4t}_{4t},
\end{align*}
\begin{align*}  
    & \E \norm {\left( \sum_{i_d =1}^n \trans{\left( \sum_{\bfi \in \calT^d_n}  \bfB_{\bfi} \cdot x_{i_1}^{(1)} \cdots x_{i_{d-1}}^{(d-1)} \right)} \left( \sum_{\bfi \in \calT^d_n} \bfB_{\bfi} \cdot x_{i_1}^{(1)} \cdots x_{i_{d-1}}^{(d-1)} \right) \right)^{1/2} }_{4t}^{4t}   \\ = & ~ \E \norm{\left(\begin{bmatrix}
        \partial_{x_1^{(d)}} \trans{\bfF}_{a,b,c} & \dots & \partial_{x_n^{(d)}} \trans{\bfF}_{a,b,c}
    \end{bmatrix}\begin{bmatrix}
        \partial_{x_1^{(d)}} \bfF_{a,b,c} \\ \vdots \\ \partial_{x_n^{(d)}} \bfF_{a,b,c} 
    \end{bmatrix}\right)^{1/2}}_{4t}^{4t} = \E \norm{\bfF_{a+1,b,c}}^{4t}_{4t}, 
\end{align*}
and
\begin{align*}
&\sum_{i_d=1}^n  \E \norm { \sum_{\bfi \in \calT^d_n}  \bfB_{\bfi} \cdot x_{i_1}^{(1)} \cdots x_{i_{d-1}}^{(d-1)}}_{4t}^{4t} \\ = & ~ \E \tr \begin{bmatrix}
\partial_{x_1^{(d)}} \bfF_{a,b,c} \cdot \partial_{x_1^{(d)}} \trans{\bfF}_{a,b,c} & & \\ & \ddots & \\ & & \partial_{x_n^{(d)}} \bfF_{a,b,c} \cdot \partial_{x_n^{(d)}} \trans{\bfF}_{a,b,c}
\end{bmatrix}^{2t} =  \E \norm{\bfF_{a,b,c+1}}^{4t}_{4t}. 
\end{align*}
Hence, the result follows.
\end{proof}

We present a moment bound for permutation symmetric, homogeneous multilinear polynomial random matrices of degree $d$. 

\begin{theorem} [Homogeneous Multilinear Recursion]
\label{homo multilinear recursion}
Let $\bfx = \{x_i\}_{i = 1}^n$ be a sequences of i.i.d random variables with $\E x_i= 0$,  $\E x_i^2= 1$ and $|x_i| \leq L$ for all $1\leq i \leq n$. Let $\{\bfA_{\bfi}\}_{\bfi \in \calT^d_n}$ be a multi-indexed sequence of deterministic matrices of the same dimension. Define a permutation symmetric, homogeneous multilinear polynomial random matrix of degree $d$ as
$$
\bfF(\bfx) =  \sum_{\bfi \in \calT^d_n} \left( \bfA_{\bfi} \cdot \prod_{j \in \{i_1,\dots, i_d \}} x_j\right).
$$
Let $a,b,c\in \Z_{\geq 0}$ and $d = a+b+c$. Then for $2 \leq t \leq \infty,$
$$
\E \norm{\bfF(\bfx) - \E \bfF(\bfx)}_{4t}^{4t} ~\leq~ \sum_{\substack{a,b,c: \\ a+b+c=d}} (48dt)^{4dt} \cdot L^{4ct} \norm{\bfF_{a,b,c}}_{4t}^{4t}.
$$
\end{theorem}
\begin{proof}
Let $\bfx^{(1)}, \dots, \bfx^{(d)}$ be $d$ independent copies of $\bfx$. We decouple $\bfF(\bfx)$ by \cref{decoupling},
$$
 E:=\E \norm{\bfF((\bfx)}_{4t}^{4t} ~\leq ~ d^{4dt} \cdot \E \norm{\bfF(\bfx^{(1)}, \dots, \bfx^{(d)})}_{4t}^{4t} .
$$
We take the partial derivative with respect to $\bfx^{(d)}$ by applying \cref{rosenthal} to get,
$$
E ~ \leq ~  d^{4dt} (16t)^{3t} \cdot \left( \E \norm{\bfF_{1,0,0}}_{4t}^{4t} + \E \norm{\bfF_{0,1,0}}_{4t}^{4t}  \right) +  d^{4dt} (8t)^{4t}L^{4t} \cdot \E \norm{\bfF_{0,0,1}}_{4t}^{4t}.
$$
Note that $\bfF_{1,0,0}$, $\bfF_{0,1,0} $, and $\bfF_{0,0,1}$ are functions in variables $\bfx^{(1)}, \dots, \bfx^{(d-1)}$. We take the partial derivative with respect to the rest of the variables until $\bfF_{a,b,c}$’s become deterministic matrices. Apply \cref{recurse rosenthal} recursively and use the permutation symmetric property of $\bfF(\bfx)$, we have
$$
E ~ \leq ~ \sum_{\substack{a,b,c: \\ a+b+c=d}} (16dt)^{4dt} \cdot L^{4ct} \frac{d!}{a!b!c!} \norm{\bfF_{a,b,c}}_{4t}^{4t}.
$$
Since $\frac{d!}{a!b!c!} \leq \frac{d!}{(d/3)!(d/3)!(d/3)!}$ and $\frac{d!}{(d/3)!(d/3)!(d/3)!} \sim \frac{3\sqrt{3}}{2\pi d}\cdot 3^d$,
$$
E ~ \leq ~ \sum_{\substack{a,b,c: \\ a+b+c=d}} (48dt)^{4dt} \cdot L^{4ct} \norm{\bfF_{a,b,c}}_{4t}^{4t}.
$$
\end{proof}

As a corollary, we derive a moment bound for general multilinear polynomial random matrices of degree $D$. The polynomials are split into homogeneous parts and each parts are bounded separately. Let $\bfF^{=d}(\bfx)$ denote the degree-$d$ homogeneous part of $\bfF(\bfx)$.

We will need the following inequality.
\begin{lemma}[Trace Inequality, see \cite{SA13} Theorem 3.1] \label{trace inequality}
    Let $\bfA_i \in \mathbb{R}^{n\times n}$ for $i=1,2,\dots, m$ and $r\geq 1$. Then
    $$
    \tr \Bigg \vert \sum_{i=1}^m \bfA_i \Bigg \vert ^r ~ \leq ~ m^{r-1} \tr \left(\sum_{i=1}^m |\bfA_i|^r\right),
    $$
    where $|\bfA_i| = (\trans{\bfA_i}\bfA_i)^{1/2}$.
\end{lemma}

\begin{corollary}[Multilinear Recursion]
\label{multilinear reursion}
Let $\bfx = \{x_i\}_{i = 1}^n$ be a sequences of i.i.d random variables with $\E x_i= 0$,  $\E x_i^2= 1$ and $|x_i| \leq L$ for all $1\leq i \leq n$. Let $\{\bfA_{\bfi}\}_{\bfi \in \calT^d_n}$ be multi-indexed sequences of deterministic matrices of the same dimension. Define a degree-$D$ multilinear polynomial random matrix as
$$
\bfF(\bfx) =  \sum_{d=1}^D ~ \sum_{\bfi \in \calT^d_n} \left( \bfA_{\bfi} \cdot \prod_{j \in \{i_1,\dots, i_d \}} x_j\right).
$$
Suppose $\bfF^{=d}(\bfx)$ is permutation symmetric for $1\leq d \leq D$. Let $a,b,c\in \Z_{\geq 0}$ and $d = a+b+c$. Then for $2 \leq t \leq \infty,$
$$
\E \norm{\bfF(\bfx) - \E \bfF(\bfx)}_{4t}^{4t} ~\leq~ \sum_{d=1}^D  D^{4t} (48dt)^{4dt} \left( \sum_{\substack{a,b,c: \\a+b+c=d }} L^{4ct} \norm{\bfF_{a,b,c}^{=d}}_{4t}^{4t} \right).
$$
\end{corollary}
\begin{proof}
We rewrite $\bfF(\bfx)$ into a formal sum of its homogeneous parts $\bfF(\bfx) = \sum_{d=1}^D \bfF^{=d} (\bfx)$. By the trace inequality (\cref{trace inequality}), we have 
$$\E \norm{\bfF(\bfx) - \E \bfF(\bfx)}_{4t}^{4t} \leq D^{4t}\sum_{d=1}^D  \E \norm{\bfF^{=d}(\bfx) -  \E \bfF^{=d}(\bfx)}_{4t}^{4t}.$$ 
By \cref{homo multilinear recursion},
$$\E \norm{\bfF(\bfx) - \E \bfF(\bfx)}_{4t}^{4t} \leq D^{4t} \sum_{d=1}^D  (48dt)^{4dt} \left( \sum_{\substack{a,b,c: \\a+b+c=d }} L^{4ct} \norm{\bfF_{a,b,c}^{=d}}_{4t}^{4t} \right).$$ 
\end{proof}

\section{Gaussian Polynomial Random Matrices}\label{sec:gaussian}
In this section, we prove a moment bound for polynomial random matrices with Gaussian random variables. While the decoupling technique could work for Gaussian polynomial random matrices as well, we base our recursion framework on the following bound for its simplicity.

\begin{lemma}[Polynomial Moments, see \cite{HT20} Theorem 7.1] \label{polynomial moments}
    Let $\bfH(\bfx) : \Omega \rightarrow \mathbb{H}^d$ be a function with $\bfx \sim \gauss{0}{\mathbb{I}_n}$. For $t=1$ and $t \geq 1.5$,
\begin{equation} \label{poly moment}
    \E \norm{\bfH(\bfx) - \E \bfH(\bfx)}_{2t}^{2t}~ \leq ~ (\sqrt{2}t)^{2t} \cdot \E \tr \left(\sum_{i=1}^n (\partial_i \bfH(\bfx) )^2 \right)^t.
\end{equation}
\end{lemma}
We derive the non-Hermitian version of \cref{polynomial moments}.
\begin{lemma}[Non-Hermitian Polynomial Moments] \label{non-Hermitian polynomial moments}
    Let $\bfF(\bfx) : \Omega \rightarrow \mathbb{R}^{d_1 \times d_2}$ be a function with $\bfx \sim \gauss{0}{\mathbb{I}_n}$. For $t=1$ and $t \geq 1.5$,
$$
\E \norm{\bfF(\bfx) - \E \bfF(\bfx) }_{2t}^{2t}~ \leq ~ (\sqrt{2}t)^{2t} \cdot \left( \E \norm{\left(\sum_{i=1}^n \partial_i \bfF(\bfx) ~ \partial_i \trans{\bfF} (\bfx) \right)^{1/2}}^{2t}_{2t} + \E \norm{\left(\sum_{i=1}^n \partial_i \trans{\bfF} (\bfx) ~ \partial_i \bfF (\bfx) \right)^{1/2}}_{2t}^{2t} \right).
$$
\end{lemma}
\begin{proof}
Apply Hermitian dilation. The left hand side of (\ref{poly moment}) becomes
$$
 \E \norm {\begin{bmatrix} {\bf 0} & \bfF - \E \bfF \\ \trans{(\bfF - \E \bfF)} & {\bf 0} \end{bmatrix}}^{2t}_{2t} = \E \tr \begin{bmatrix} (\bfF - \E \bfF) \trans{(\bfF - \E \bfF)} &  {\bf 0} \\ {\bf 0} & \trans{(\bfF - \E \bfF)} (\bfF - \E \bfF) \end{bmatrix}^{t} = 2 \E \norm{\bfF - \E \bfF}_{2t}^{2t},
$$  
and the right hand side of (\ref{poly moment}) becomes
\begin{align*}
   \E \tr \left(\sum_{i=1}^n \begin{bmatrix}
  {\bf 0} &  \partial_i \bfF \\ \partial_i \trans{\bfF} & {\bf 0} 
\end{bmatrix}^2 \right)^t &= \E \tr \begin{bmatrix}
  \sum_{i=1}^n \partial_i \bfF ~ \partial_i \trans{\bfF} & {\bf 0} \\ {\bf 0}  & \sum_{i=1}^n \partial_i \trans{\bfF} ~ \partial_i \bfF
\end{bmatrix}^t \\
&=  \E \norm{\left(\sum_{i=1}^n \partial_i \bfF ~ \partial_i \trans{\bfF} \right)^{1/2}}^{2t}_{2t} + \E \norm{\left(\sum_{i=1}^n \partial_i \trans{\bfF} ~ \partial_i \bfF \right)^{1/2}}_{2t}^{2t}.
\end{align*}
Combine both sides and the result follows.
\end{proof}

\begin{theorem}[Gaussian Recursion] \label{gaussian recursion}
 Let $\bfx \sim \gauss{0}{\mathbb{I}_n}$ and $\{\bfA_{\bfi}\}_{\bfi \in [n]^d}$ be multi-indexed sequences of deterministic matrices of the same dimension. Define a degree-$d$ homogeneous Gaussian polynomial random matrix as
$$
\bfP(\bfx) =  \sum_{\bfi \in [n]^d} \left( \bfA_{\bfi} \cdot \prod_{j \in \{i_1,\dots, i_d \}} x_j\right).
$$
Let $a,b \in \Z_{\geq 0}$. Then for $2 \leq t \leq \infty,$
$$
\E \norm{\bfP(\bfx) - \E \bfP(\bfx)}_{2t}^{2t} ~\leq~ (2^d \sqrt{2}t)^{2t} \left( \sum_{1 \leq a+b \leq d}\norm{\E \bfP_{a,b}}_{2t}^{2t} \right).
$$   
\end{theorem}
\begin{proof} 
Start by applying \cref{gaussian recursion claim} with $k = 0$. Apply \cref{gaussian recursion claim} recursively until $k = d-1$ to get the desired bound.
\end{proof}

\begin{claim} \label{gaussian recursion claim}
 Let $\bfx \sim \gauss{0}{\mathbb{I}_n}$ and $\{\bfA_{\bfi}\}_{\bfi \in [n]^d}$ be multi-indexed sequences of deterministic matrices of the same dimension. Define a degree-$d$ homogeneous Gaussian polynomial random matrix as
$$
\bfP(\bfx) =  \sum_{\bfi \in [n]^d} \left( \bfA_{\bfi} \cdot \prod_{j \in \{i_1,\dots, i_d \}} x_j\right).
$$
Let $a,b \in \Z_{\geq 0}$ and $a+b = k < d$. Let $\bfP_{a,b}(\bfx)$ be the $k$-th partial derivative block matrix associated to $\bfP(\bfx)$. Then for $2 \leq t \leq \infty,$
\begin{align*}
\E \norm{\bfP_{a,b} - \E \bfP_{a,b}}_{2t}^{2t} ~&\leq ~ (2\sqrt{2}t)^{2t} \cdot \left( \E \norm{\bfP_{a+1,b} - \E \bfP_{a+1,b}}^{2t}_{2t}  + \norm{\E \bfP_{a+1,b}}^{2t}_{2t} \right) \\
 & ~ + (2\sqrt{2}t)^{2t} \cdot \left( \E \norm{\bfP_{a,b+1} - \E \bfP_{a,b+1}}^{2t}_{2t} + \norm{\E \bfP_{a,b+1}}^{2t}_{2t} \right).
\end{align*}    
\end{claim}
\begin{proof}
 By the non-hermitian polynomial moment (\cref{non-Hermitian polynomial moments}), we have
\begin{align*}
E:=\E \norm{\bfP_{a,b} - \E \bfP_{a,b}}_{2t}^{2t} & \leq (\sqrt{2}t)^{2t} \cdot  \left( \E \norm{\left(\sum_{i=1}^n \partial_{i} \bfP_{a,b} ~ \partial_{i} \trans{\bfP}_{a,b}\right)^{1/2}}^{2t}_{2t} + \E \norm{\left(\sum_{i=1}^n \partial_{i} \trans{\bfP}_{a,b} ~ \partial_{i} \bfP_{a,b}\right)^{1/2}}_{2t}^{2t} \right) \\
& =  (\sqrt{2}t)^{2t} \cdot \left( \E \norm{\begin{bmatrix}
    \partial_1 \bfP_{a,b} & \dots & \partial_n \bfP_{a,b}
\end{bmatrix}}^{2t}_{2t} + \E \norm{\begin{bmatrix}
    \partial_1 \bfP_{a,b} \\ \vdots \\ \partial_n \bfP_{a,b}
\end{bmatrix}}^{2t}_{2t}  \right) \\
& = (\sqrt{2}t)^{2t} \cdot \left(\E \norm{\bfP_{a+1,b}}^{2t}_{2t} + \E \norm{\bfP_{a,b+1}}^{2t}_{2t} \right), 
\end{align*}
where the last equality uses our notation for partial derivative block matrices introduced in \cref{partial}. Notice that $\bfP_{a+1,b}$ and $\bfP_{a,b+1}$ are homogeneous Gaussian polynomial random matrix of degree $k-1$. Since $\bfP_{a+1,b}$ and $\bfP_{a,b+1}$ are not necessarily centered, 
\begin{align*}
E ~ & \leq ~ (\sqrt{2}t)^{2t} \cdot \left(\E \norm{\bfP_{a+1,b} - \E \bfP_{a+1,b} + \E \bfP_{a+1,b}}^{2t}_{2t} + \E \norm{\bfP_{a,b+1} - \E \bfP_{a,b+1} +\E \bfP_{a,b+1}}^{2t}_{2t} \right) \\
& \leq ~ (2\sqrt{2}t)^{2t} \cdot \left(\E \norm{\bfP_{a+1,b} - \E \bfP_{a+1,b}}^{2t}_{2t} + \norm{\E \bfP_{a+1,b}}^{2t}_{2t} + \E \norm{\bfP_{a,b+1} - \E \bfP_{a,b+1}}^{2t}_{2t} + \norm{\E \bfP_{a,b+1}}^{2t}_{2t} \right),
\end{align*}
where the last inequality is due to the trace inequality (\cref{trace inequality}). 
\end{proof}

\section{Applications}\label{sec:applications}
\subsection{Graph Matrices}
Denote the \Erdos-\Renyi random graph on $n$ vertices as $\calG_{n,p}$, where each edge is present with probability $p$ independent of all other edges. When $\calG_{n,p}$ is viewed as a probability space, it is equal to $(\Omega, \cal F, \mu)$ where $\Omega$ is the sample space of all possible graphs on $n$ vertices. Any $G \in \Omega$ is a random vector representing all edges. Each coordinate in $G$, denoted by $G_{ij}$, is an independent random variable representing a single edge. We can construct a random graph by drawing a copy of $G \sim \mu$. To adopt the convention in $p$-biased Fourier analysis, we normalize $G_{ij}$ such that $\E G_{ij} = 0$ and $\E G_{ij}^2 = 1$, which leads to the sample space $\Omega = \left \{ - \sqrt{\frac{1-p}{p}}, \sqrt{\frac{p}{1-p}} \right\}^{\binom{n}{2}}$.

\begin{definition} [Shape, see e.g. \cite{RT23} Definition 4.2]
A shape is a tuple $\tau = (V(\tau), E(\tau), U_{\tau}, V_{\tau})$ where $(V(\tau), E(\tau))$ is a graph and $U_{\tau}$, $V_{\tau}$ are ordered subsets of the vertices.
\end{definition}

\begin{definition} [Graph matrix, see e.g. \cite{RT23} Definition 4.4]
Given a shape $\tau$, the associated graph matrix $\bfM$ : $\Omega \rightarrow \R^{d_1 \times d_2}$ is a matrix-valued function such that
$$
\bfM [I,J] = \sum_{\substack{\phi \in \calT^k_n: \\ \phi(U_\tau) = I, \phi(V_\tau) = J }} \prod _{(u,v) \in E(\tau)} G_{\phi(u),\phi(v)}.
$$
In other words, $\bfM$ maps an input graph $G \in \Omega$ to a $\frac{n!}{(n-|I|)!} \times \frac{n!}{(n-|J|)!}$ matrix whose rows and columns are indexed by $I$ and $J$ respectively.
\end{definition}

\begin{definition}[Vertex Separator, see e.g. \cite{RT23} Definition 4.8]
    For any shape $\tau$, a vertex separator is a subset of vertices $S \subseteq V(\tau)$ such that there is no path from $U_\tau$ to $V_\tau$ in $\tau ~ \textbackslash ~ S$, which is the shape obtained by deleting $S$ and all edges adjacent to S. We write $S_\tau$ for a vertex separator of the minimum size.
\end{definition}

\begin{claim} \label{vertex separator}
   For any shape $\tau$, let $\{E_1, E_2\}$ be an arbitrary cover of $E(\tau)$. Let $S=E_1 \cap E_2$ be the vertex set of the overlapping edges between $E_1$ and $E_2$. If $S$ contains $E_1 \cap V_\tau$ and $E_2 \cap U_\tau$, then $S$ is a vertex separator. 
\end{claim}

\begin{proof}
    Since $S = E_1 \cap E_2$, there is no path from $E_1 \backslash S$ to $E_2 \backslash S$. Additionally, if $S$ contains $E_1 \cap V_\tau$, then $E_1 \backslash S$ is not adjacent to $V_\tau$ and $E_1 \backslash S$ contains $U_\tau \backslash S$. If $S$ also contains $E_2 \cap U_\tau$, then $E_2 \backslash S$ is not adjacent to $U_\tau$ and $E_2 \backslash S$ contains $V_\tau \backslash S$. Since there is no path from $U_\tau \backslash S$ to $V_\tau \backslash S$, $S$ is a vertex separator.
\end{proof}

\begin{theorem} \label{graph matrix}
For any shape $\tau$ and $2 \leq t \leq \infty$,
$$
\E \norm{\bfM}_{4t}^{4t} ~\leq~ \left((48t|V(\tau)|)^{4t|V(\tau)|} (C|E(\tau)|)^{|E(\tau)|} n^{|V(\tau)|} \right) \left( \sqrt{\frac{1-p}{p}}\right)^{4t|E(S_\tau)|} n^{2t(|V(\tau)| - |S_\tau|)}
$$
where $C$ is an absolute constant and $E(S_\tau)$ are all edges adjacent to $S_\tau$.
\end{theorem}

\begin{proof}
First, let $|V(\tau)| = k$ and we write $\bfM$ as a polynomial whose coefficients are matrices,
\begin{equation*} 
   \bfM = \sum_{\phi \in \calT^k_n} \bfB_{\phi(E)} \prod _{(i,j) \in E(\tau)} G_{\phi(i),\phi(j)}, 
\end{equation*}
where the $[\textit{I}, \textit{J}]$-entry of $\bfB_{\phi(E)}$ is
\begin{equation*}
\bfB_{\phi(E)}[\textit{I}, \textit{J}] =
\begin{cases}
1 & \text{if} ~ \phi(U_\tau) = \textit{I} ~~ \text{and} ~~ \phi(V_\tau) = \textit{J} \\
0 & \text{otherwise}.
\end{cases} 
\end{equation*}
Note that $\bfB_{\phi(E)}$’s have the same dimension as $\bfM$ and the rows and columns of $\bfB_{\phi(E)}$’s are indexed in the same way as $\bfM$. Since $U_\tau$ and $V_\tau$ are ordered and each $\phi$ assigns one set of value to vertices in $U_\tau$ and $V_\tau$, there is only one nonzero entry in each $\bfB_{\phi(E)}$. But there might be multiple $\phi$’s for which $\bfB_{\phi(E)}$’s are identical due to free vertices.

Next, let $|E(\tau)| = d$ and we rewrite $\bfM$ in a permutation symmetric way,
\begin{align*}
  \bfM & = \sum_{\phi \in \calT^k_n} \left(\frac{1}{d!}\sum_{\sigma \in \mathfrak{S}_d} \bfB_{\sigma(\phi(E))} G_{\phi(i_{\sigma(1)}),\phi(j_{\sigma(1)})} ~ \cdots ~ G_{\phi(i_{\sigma(d)}),\phi(j_{\sigma(d)})} \right) \\
  & = \frac{1}{d!}\sum_{\phi \in \calT^k_n} \bfA_{\phi(E)} \cdot G_{\phi(i_1),\phi(j_1)} ~ \cdots ~ G_{\phi(i_d),\phi(j_d)},
\end{align*}
where $\bfA_{\phi(E)} = \sum_{\sigma \in \mathfrak{S}_d} \bfB_{\sigma(\phi(E))} $. Notice that the indices of the random variables of $\bfM$ have the same structure as in \cref{graph_decoupling}, so we can decouple $\bfM$ using \cref{graph_decoupling},
$$
F := \E \norm{\bfM}_{4t}^{4t} ~\leq ~ |V(\tau)|^{4t|V(\tau)|} \left(\frac{1}{d!}\right)^{4t} \cdot \E \norm{\sum_{\phi \in \calT^k_n} \bfA_{\phi(E)} \cdot G_{\phi(i_1), \phi(j_1)}^{(1)}\cdots G_{\phi(i_d), \phi(j_d)}^{(d)}}_{4t}^{4t}, 
$$
where $\{G^{(1)}, \ldots, G^{(d)}\}$ denote $d$ independent copies of $G$. Since $\bfM$ is a multilinear polynomial random matrix, $\E \bfM = \bf0$. An application of \cref{homo multilinear recursion} yields,
\begin{equation*}
    F  ~ \leq ~ \sum_{\substack{a,b,c: \\ a+b+c=|E(\tau)|}} (48t|V(\tau)|)^{4t|V(\tau)|} \left(\frac{1}{|E(\tau)|!} \right)^{4t} \left( \sqrt{\frac{1-p}{p}}\right)^{4ct} \norm{\bfM_{a,b,c}}_{4t}^{4t},
\end{equation*}
where $\{\bfM_{a,b,c}\}$ are partial derivative block matrices associated to $\bfM$.

Fix a set of $a,b,c$, $\bfM_{a,b,c}$ is a block matrix whose blocks are comprised of $\{\bfA_{\phi(E)}\}_{\phi \in \calT^k_n}$. More specifically, the $[I,J]$-th block of $\bfM_{a,b,c}$ is
\begin{equation} \label{block entry}
\bfM_{a,b,c}[I,J] = \partial_{G_{I_1}} \cdots \partial_{G_{I_{a+c}}} \partial_{G_{J_1}} \cdots \partial_{G_{J_{b}}} \bfM 
    = \bfA_{I \cup J} = \sum_{\sigma \in \mathfrak{S}_d} \bfB_{\sigma(I \cup J)}.
\end{equation}

Algebraically, the last equality is due to the commutativity of partial derivative operators. Combinatorially, if we fix an arbitrary ordering on the edge set $E(\tau)$, the last expression is summing up all permutations on $d$ edges in $E(\tau)$. Now let’s delve into the combinatorics. Let $E(\tau) = E_1 \cup E_2$, where $E_1$ contains $a+c$ edges and $E_2$ contains $b+c$ edges with $c$ number of overlapping edges. Let $S = E_1 \cap E_2$ be the set of vertices shared by $E_1$ and $E_2$. By permuting the edges in $E(\tau)$, $\{\sigma(E_1), \sigma(E_2)\}$ are all possible covers of $E(\tau)$ by two overlapping sets of ordered edges. The row and column blocks of $\bfM_{a,b,c}$ are indexed by $\phi(E_1)$ and $\phi(E_2)$ respectively.

Let’s write $$\bfM_{a,b,c} = \sum_{\sigma \in \mathfrak{S}_d}\bfM_{a,b,c,\sigma},$$ where $\bfM_{a,b,c,\sigma}[I,J] = \bfB_{\sigma(I \cup J)}$. By the trace inequality (\cref{trace inequality}), $$\norm{\bfM_{a,b,c}}^{4t}_{4t} \leq \left(d!\right)^{4t}\sum_{\sigma \in \mathfrak{S}_d} \norm{\bfM_{a,b,c,\sigma}} ^{4t}_{4t}.$$ It follows that 
\begin{equation}\label{F}
    F  ~ \leq ~ \sum_{\substack{a,b,c: \\ a+b+c=|E(\tau)|}} \sum_{\sigma \in \mathfrak{S}_d} (48t|V(\tau)|)^{4t|V(\tau)|} \left( \sqrt{\frac{1-p}{p}}\right)^{4ct} \norm{\bfM_{a,b,c, \sigma}}_{4t}^{4t}.
\end{equation}
For any two distinct $\sigma_1$ and $\sigma_2 \in \mathfrak{S}_d$, $\norm{\bfM_{a,b,c,\sigma_1}}^{4t}_{4t} = \norm{\bfM_{a,b,c,\sigma_2}}^{4t}_{4t}$ since $\bfM_{a,b,c,\sigma_1}$ and $\bfM_{a,b,c,\sigma_2}$ are identical after permuting rows and columns. So we will focus on bounding $\norm{\bfM_{a,b,c,\sigma_0}}^{4t}_{4t}$ where $\sigma_0$ is the identity permutation. For identity permutation, we have $\bfM_{a,b,c, \sigma_0}[I,J] = \bfB_{I \cup J}$ by (\ref{block entry}). Denote $\bfM_{a,b,c,\sigma_0}$ by $\bfF$, 
\begin{align}
    E := &\norm{\bfM_{a,b,c,\sigma_0}}^{4t}_{4t} = \norm{\bfF}^{4t}_{4t} =  \tr \left(\trans{\bfF} \bfF \right)^{2t} \notag \\
    =& \tr \sum_{I_1,\dots,I_{2t} \in [n]^{2(a+c)}} \sum_{J_1,\dots,J_{2t} \in [n]^{2(b+c)}} \trans{\bfF}[I_1, J_1]~\bfF[I_1, J_2]~\trans{\bfF}[I_2, J_2]~\bfF [I_2, J_3] \cdots \trans{\bfF}[I_{2t}, J_{2t}]~\bfF [I_{2t}, J_1] \notag \\
    =& \tr \sum_{I_1,\dots,I_{2t} \in [n]^{2(a+c)}} \sum_{J_1,\dots,J_{2t} \in [n]^{2(b+c)}}  \trans{\bfB}_{I_1 \cup J_1}~\bfB_{I_1 \cup J_2}~\trans{\bfB}_{I_2 \cup J_2}~\bfB_{I_2 \cup J_3} \cdots \trans{\bfB}_{I_{2t} \cup J_{2t}}~\bfB_{I_{2t} \cup J_1} \label{summands}
\end{align}

We take a look at for which $I_1, \dots I_{2t}$ and $J_1, \dots, J_{2t}$ that the summands in (\ref{summands}) are nonzero and the structure of these nonzero summands. First of all, $\bfB_{I_i \cup J_i} \neq {\bf 0}$ implies that $I_i \cup J_i = \phi(E_1 \cup E_2)$, for some $\phi \in \calT^k_n$ and $I_i(S) = J_i(S)$ for $1\leq i \leq 2t$. Similarly, $\bfB_{I_i \cup J_{i+1}} \neq {\bf 0}$ implies that $I_i \cup J_{i+1} = \phi’(E_1 \cup E_2)$, for some $\phi’ \in \calT^k_n$ and $I_i(S) = J_{i+1}(S)$ for $1\leq i \leq 2t$ (the additions in the subscripts are in mod $2t, \ie J_{2t+1} = J_1$). Secondly, for each $\bfB_{I_i \cup J_i} \neq {\bf 0}$, there is only one nonzero entry, namely $\bfB_{I_i \cup J_i}[I_i \cup J_i(U_\tau), I_i \cup J_i(V_\tau)] = 1$. So $\trans{\bfB}_{I_i \cup J_i}\bfB_{I_i \cup J_{i+1}} \neq {\bf 0}$ if $I_i \cup J_i(U_\tau) = I_i \cup J_{i+1}(U_\tau)$ and the only nonzero entry is 
$$
\trans{\bfB_{I_i \cup J_i}}\bfB_{I_i \cup J_{i+1}} [I_i \cup J_i(V_\tau), I_i \cup J_{i+1}(V_\tau)] = 1.
$$
It follows that $\trans{\bfB}_{I_i \cup J_i} \bfB_{I_i \cup J_{i+1}} \neq {\bf 0}$ for $1\leq i \leq 2t$ if $J_1(E_2 \cap U_\tau) = J_2(E_2 \cap U_\tau) = \cdots = J_{2t}(E_2 \cap U_\tau)$. Lastly, $\trans{\bfB}_{I_i \cup J_i}\bfB_{I_i \cup J_{i+1}} \trans{\bfB}_{I_{i+1} \cup J_{i+1}}\bfB_{I_{i+1} \cup J_{i+2}}\neq {\bf 0}$ if $I_i \cup J_{i+1}(V_\tau) = I_{i+1} \cup J_{i+1}(V_\tau)$ and the only nonzero entry is 
$$
\trans{\bfB}_{I_i \cup J_i}\bfB_{I_i \cup J_{i+1}} \trans{\bfB}_{I_{i+1} \cup J_{i+1}}\bfB_{I_{i+1} \cup J_{i+2}}[I_{i} \cup J_{i}(V_\tau), I_{i+1} \cup J_{i+2}(V_\tau)] = 1.
$$
It follows that $\trans{\bfB}_{I_1 \cup J_1}~\bfB_{I_1 \cup J_2}~\trans{\bfB}_{I_2 \cup J_2}~\bfB_{I_2 \cup J_3} \cdots \trans{\bfB}_{I_{2t} \cup J_{2t}}~\bfB_{I_{2t} \cup J_1} \neq {\bf 0}$ if $I_1(E_1 \cap V_\tau) = I_2(E_1 \cap V_\tau) = \cdots = I_{2t}(E_1 \cap V_\tau)$ and the only nonzero entry is 
$$
\trans{\bfB}_{I_1 \cup J_1}~\bfB_{I_1 \cup J_2}~\trans{\bfB}_{I_2 \cup J_2}~\bfB_{I_2 \cup J_3} \cdots \trans{\bfB}_{I_{2t} \cup J_{2t}}~\bfB_{I_{2t} \cup J_1} [I_1 \cup J_{1}(V_\tau), I_{2t} \cup J_{1}(V_\tau)] = 1.
$$
Since $I_1(E_1 \cap V_\tau) = I_{2t}(E_1 \cap V_\tau)$, we have $I_1 \cup J_{1}(V_\tau) = I_{2t} \cup J_{1}(V_\tau)$. So if the summand is nonzero, it has a $1$ on its diagonal. 

To summarize, each summand in (\ref{summands}) is nonzero if and only if $I_i \cup J_i = \phi_i(E_1 \cup E_2), I_i \cup J_{i+1} =\phi’_i(E_1 \cup E_2)$ for some $\phi_i, \phi’_i\in \calT^k_n$, $I_i$’s and $J_i$’s agree on $S$, $I_i$’s agree on $E_1\cap V_\tau$ and $J_i$’s agree on $E_2\cap U_\tau$ for $1 \leq i \leq 2t$. Since each nonzero summand contributes a $1$ on the diagonal, we can simply count the number of nonzero summands to compute the trace in (\ref{summands}). Notice that the number of nonzero summands is equal to the number of $\phi_1, \dots, \phi_{2t}, \phi’_1, \dots, \phi’_{2t} $ that satisfy all the constraints. Hence 
\begin{equation} \label{upper bound1}
E \leq n^{|S|} ~ n^{|E_1 \cap V_\tau|} ~ n^{|E_2 \cap U_\tau|} ~ n^{2t|E_1 \backslash V_\tau \backslash S|}  ~ n^{2t|E_2 \backslash U_\tau \backslash S|}.    
\end{equation}

For a fixed set of $a,b,c$, (\ref{upper bound1}) provides an upper bound for $\norm{\bfM_{a,b,c, \sigma}}^{4t}_{4t}$ for any permutation $\sigma \in \mathfrak{S}_d$. But what is an upper bound for $\norm{\bfM_{a,b,c, \sigma}}^{4t}_{4t}$ among all possible sets of $a,b,c$? By varying $a,b,c$, we can always make $S = E_1 \cap E_2$ large enough to contain vertices in $E_1 \cap V_\tau$ and $E_2 \cap U_\tau$. Thus $S$ is a vertex separator by \cref{vertex separator}. It follows that $E_1 \backslash V_\tau \backslash S = E_1 \backslash S$ and $E_2 \backslash U_\tau \backslash S = E_2 \backslash S$. Hence
\begin{equation} \label{upper bound2}
  E ~\leq~ n^{|V(\tau)|} ~ n^{2t|E_1 \backslash S|}  ~ n^{2t|E_2 \backslash S|} ~\leq~ n^{|V(\tau)|}~ n^{2t(|V(\tau)| - |S_\tau|)}, 
\end{equation}
where $S_\tau$ is a vertex separator of minimum size. Substituting (\ref{upper bound2}) into (\ref{F}) yields the desired bound.
\end{proof}

\begin{corollary} For any given shape $\tau$, any $\epsilon > 0$, with probability $1-\epsilon$,
    $$
     \norm{\bfM} ~ \leq ~ |V(\tau)|^{|V(\tau)|} \left( C \log \left(|E(\tau)|^{|E(\tau)|} n^{|V(\tau)|} / \epsilon \right) \right)^{|V(\tau)|} \left( \sqrt{\frac{1-p}{p}}\right)^{|E(S_\tau)|} \sqrt{n}^{|V(\tau)| - |S_\tau|},
    $$
    where $C >0$ is an absolute constant.
\end{corollary}
\begin{proof}
    By Markov’s inequality and \cref{graph matrix}, we have
    \begin{align}
        &\Psymb( \norm{\bfM}  \geq \theta) ~\leq~ \Psymb(\norm{\bfM}^{4t}_{4t} \geq \theta^{4t}) ~\leq ~ \theta^{-4t}  \E \norm{\bfM}_{4t}^{4t} \notag \\
        = ~ & \theta^{-4t} \left((48t|V(\tau)|)^{4t|V(\tau)|}(C|E(\tau)|)^{|E(\tau)|} n^{|V(\tau)|} \right) \left( \sqrt{\frac{1-p}{p}}\right)^{4t|E(S_\tau)|} n^{2t(|V(\tau)| - |S_\tau|)}. \label{upper bound3}
    \end{align}
    Set the right hand side of (\ref{upper bound3}) to $\epsilon$ by taking
    $$
    \theta = \left( \epsilon^{-\frac{1}{4t}}(48t|V(\tau)|)^{|V(\tau)|} (C|E(\tau)|)^{\frac{|E(\tau)|}{4t}} n^{\frac{|V(\tau)|}{4t}} \right) \left( \sqrt{\frac{1-p}{p}}\right)^{|E(S_\tau)|} \sqrt{n}^{|V(\tau)| - |S_\tau|}.
    $$
    Take
    $$
    t = \frac{1}{4} \log \left(|E(\tau)|^{|E(\tau)|} n^{|V(\tau)|} / \epsilon \right)
    $$
    to complete the proof.
\end{proof}

\subsection{A Simple Tensor Network}
A tensor network is a diagram of a collection of tensors connected together by contractions, which results in another tensor, matrix, vector or scalar. Many existing algorithms for tensor decomposition and tensor PCA analyze spectral properties of matrices arising from tensor networks \cite{RM14, HSS15, HSSS16, MW19, DDOCHST22, OTR22}. In particular, the melon graph (\cref{melon}) is similar to the tensor unfolding method used in \cite{RM14}. The spectral norm of the matrices associated to the melon graph is studied in \cite{OTR22} using combinatorial tools developed in Random Tensor Theory \cite{Gur16}. In this section, we will demonstrate our recursion framework for Gaussian polynomial random matrices (\cref{gaussian recursion}) in analyzing the spectral norm of the pure noise matrix associated to the melon graph.

\begin{figure}[h!]
    \centering
    \includegraphics[width=0.45\linewidth]{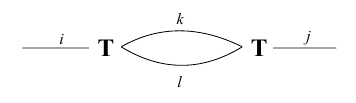}
    \caption{The Melon Graph}
    \label{melon}
\end{figure}

\begin{theorem}(See also \cite{Oue22:thesis} Theorem 7) Let $\bfT \in (\mathbb{R}^n)^{\otimes 3}$ be a random noise tensor with i.i.d. standard Gaussian entries. For $i,j \in [n]$, let $T_i$ and $T_j$ be the first-mode slices of $\bfT$. The pure noise matrix $\bfM \in \mathbb{R}^{n \times n}$ associated to the melon graph has entries $\bfM[i,j] = \langle T_i, T_j \rangle$. Then
$$
\E \norm{\bfM - \E \bfM}_{2t}^{2t} ~ \leq ~ \calO (n^{3t}).
$$
\end{theorem}
\begin{proof}
Since $\bfM$ is a degree 2 homogeneous Gaussian random matrix, $\E \bfM_{1,0} = \E \bfM_{0,1} = \bf 0$. By \cref{gaussian recursion}, 
$$
\E \norm{\bfM - \E \bfM}_{2t}^{2t} ~\leq~ (8t)^{2t} \left(\norm{\bfM_{2,0}}_{2t}^{2t} + \norm{\bfM_{0,2}}_{2t}^{2t} + \norm{\bfM_{1,1}}_{2t}^{2t} \right).
$$
Let $\{\bfB_{i,j,k,l}\}$ be the second order partial derivatives of $\bfM$, \ie
$$
\bfB_{i,j,k,l}  = \partial_{\bfT_{i,k,l}} ~ \partial_{\bfT_{j,k,l}} \bfM.
$$
Note that $\bfM_{2,0}$, $\bfM_{0,2}$ and $\bfM_{1,1}$ are made up with $\{\bfB_{i,j,k,l}\}$ stacked in different ways.
Let $E_{i,j}$ be the matrix basis with $1$ in its $[i,j]$-th entry and $0$ otherwise. For $i = j$, 
$$
\bfB_{i,i,k,l} ~ = ~ \partial^2_{\bfT_{i,k,l}} \bfM ~ = ~2 \cdot E_{i,i},
$$
For $i \neq j$, 
$$
\bfB_{i,j,k,l} ~ = ~  \partial_{\bfT_{i,k,l}}~\partial_{\bfT_{j,k,l}} \bfM  ~ = ~ E_{i,j} + E_{j,i}.
$$
Hence,
\begin{align*}
\norm{\bfM_{0,2}}^{2t}_{2t} & = \tr \left( \trans{\bfM}_{0,2} ~ \bfM_{0,2}\right)^t \\
& = \tr \left( \sum_{i,j,k,l \in [n]} \trans{\bfB}_{i,j,k,l} \bfB_{i,j,k,l} \right)^t \\
& = \tr \left( \sum_{i=j,k,l \in [n]} \trans{\bfB}_{i,j,k,l} \bfB_{i,j,k,l}  +  \sum_{i \neq j,k,l \in [n]} \trans{\bfB}_{i,j,k,l} \bfB_{i,j,k,l}\right)^t \\
& = \tr \left( \sum_{i,k,l \in [n]} 2 \cdot E_{i,i}  +  \sum_{i \neq j,k,l \in [n]} \left(E_{i,i} + E_{j,j} \right) \right)^t \\
& \leq \tr \left( 2 n^2 \cdot \mathbb{I}_n  +  n^3  \cdot \mathbb{I}_n \right)^t \\
& = \calO (n^{3t}).
\end{align*}

Since $\bfM$ is hermitian, $\trans{\bfB}_{i,j,k,l} = \bfB_{i,j,k,l} $, so $\norm{\bfM_{0,2}}^{2t}_{2t} = \norm{\bfM_{2,0}}^{2t}_{2t}$. We proceed to bound $\norm{\bfM_{1,1}}_{2t}^{2t}$. Since permutation of rows and columns doesn’t affect Schatten norm, we permute the rows and columns of $\bfM_{1,1}$ into a block diagonal matrix with its diagonal block $\bfD_{k,l}$ indexed by $k$ and $l$. 
$$
E : = \norm{\bfM_{1,1}}_{2t}^{2t} = \tr\left( \trans{\bfM}_{1,1} \bfM_{1,1} \right)^t =  \sum_{k,l \in [n]} \tr \left( \trans{\bfD}_{k,l} \bfD_{k,l} \right)^t \leq n^2 \cdot \norm{\bfD_{k,l}}^{2t}_{2t} \leq n^2 \cdot \left(\norm{\bfD_{k,l}}^{2}_{2} \right)^t,
$$
where we used the Frobenius norm to upper bound the Schatten norm in the last inequality. $\bfD_{k,l}$ is a block matrix with $\bfD_{k,l}[i,i] = 2 \cdot E_{i,i}$ and $\bfD_{k,l}[i,j] = E_{i,j} + E_{j,i}$. So $\norm{\bfD_{k,l}}^{2}_{2} \leq 4n + 2n^2$. Thus
$$
E ~ \leq ~ n^2 (4n + 2n^2 )^t ~\leq ~ \calO(n^{2t}).
$$
Hence $\norm{\bfM_{0,2}}^{2t}_{2t} \leq \calO(n^{3t})$ dominates and the result follows.
\end{proof}



%

\bibliographystyle{alphaurl}
\bibliography{macros,madhur}

\end{document}

%% file: macros.tex
\usepackage{xspace,enumerate}
\usepackage{amsmath,amssymb}
\usepackage{amsthm}
\usepackage[toc,page]{appendix}
\usepackage{thmtools}
\usepackage{thm-restate}
\usepackage{color,graphicx}
\usepackage{boxedminipage}
\usepackage{makecell}
\usepackage{tabularx}

\ifnum\showkeys=1
\usepackage[color]{showkeys}
\fi

\definecolor{darkred}{rgb}{0.5,0,0}
\definecolor{darkgreen}{rgb}{0,0.35,0}
\definecolor{darkblue}{rgb}{0,0,0.55}

\usepackage[dvipsnames]{xcolor}

\usepackage[pdfstartview=FitH,pdfpagemode=None,colorlinks,linkcolor=NavyBlue,filecolor=blue,citecolor=OliveGreen,urlcolor=NavyBlue,pagebackref]{hyperref}

\usepackage[capitalise,nameinlink]{cleveref}
\usepackage[T1]{fontenc}
\usepackage{mathtools,dsfont,bbm}
\usepackage{mathpazo}
\usepackage{microtype}
\ifnum\widemargin=0
\usepackage[top=1in, bottom=1in, left=1in, right=1in]{geometry}
\else
\usepackage[top=1in, bottom=1in, left=1.25in, right=1.25in]{geometry}
\fi

\ifnum\showauthornotes=1
\newcommand{\Authornote}[3]{{\sf\small\color{#3}{[#1: #2]}}}
\newcommand{\Authorcomment}[2]{{\sf \small\color{gray}{[#1: #2]}}}
\newcommand{\Authorfnote}[2]{\footnote{\color{red}{#1: #2}}}
\else
\newcommand{\Authornote}[3]{}
\newcommand{\Authorcomment}[2]{}
\newcommand{\Authorfnote}[2]{}
\fi

\ifnum\showdraftbox=1

\else

\fi

\newtheorem{theorem}{Theorem}[section]

\newtheorem{definition}[theorem]{Definition}

\newtheorem{lemma}[theorem]{Lemma}
\newtheorem{remark}[theorem]{Remark}

\newtheorem{corollary}[theorem]{Corollary}
\newtheorem{claim}[theorem]{Claim}

\newtheorem{example}[theorem]{Example}

\newtheorem{algo}[theorem]{Algorithm}

\def\FullBox{\hbox{\vrule width 6pt height 6pt depth 0pt}}

\def\qed{\ifmmode\qquad\FullBox\else{\unskip\nobreak\hfil
\penalty50\hskip1em\null\nobreak\hfil\FullBox
\parfillskip=0pt\finalhyphendemerits=0\endgraf}\fi}

\def\qedsketch{\ifmmode\Box\else{\unskip\nobreak\hfil
\penalty50\hskip1em\null\nobreak\hfil$\Box$
\parfillskip=0pt\finalhyphendemerits=0\endgraf}\fi}

\def\to{\rightarrow}

\def\epsilon{\varepsilon}

\def\phi{\varphi}
\def\cal{\mathcal}

\def\implies{\Rightarrow}

\newcommand{\given}{\;\ifnum\currentgrouptype=16 \middle\fi \vert\;}

\newcommand{\ie}{i.e.,\xspace}

\newcommand{\etal}{et al.\xspace}

\newcommand{\mper}{\,.}
\newcommand{\mcom}{\,,}

\newcommand{\R}{{\mathbb R}}
\newcommand{\E}{{\mathbb E}}

\newcommand{\Z}{{\mathbb Z}}

\newcommand{\indicator}[1]{\mathds{1}_{\{#1\}}}

\newcommand{\gauss}[2]{{\cal N(#1, #2)}}

\usepackage{nicefrac}

\newcommand{\abs}[1]{\ensuremath{\left\lvert #1 \right\rvert}}

\newcommand{\norm}[1]{\ensuremath{\left\lVert #1 \right\rVert}}

\def\bfx {{\bf x}}
\def\bfy {{\bf y}}
\def\bfz {{\bf z}}
\def\bfi {{\bf i}}
\def\bfj {{\bf j}}

\def\bfA{{\bf A}}
\def\bfB{{\bf B}}
\def\bfD{{\bf D}}
\def\bfF{{\bf F}}
\def\bfH{{\bf H}}
\def\bfM{{\bf M}}
\def\bfP{{\bf P}}
\def\bfT{{\bf T}}
\def\bfX {{\bf X}}
\def\bfY {{\bf Y}}

\newcommand{\Esymb}{\mathbb{E}}
\newcommand{\Hsymb}{\mathbb{H}}
\newcommand{\Psymb}{\mathbb{P}}

\newcommand{\Varsymb}{\mathrm{Var}}
\DeclareMathOperator*{\ExpOp}{\Esymb}

\makeatletter
\def\Pr#1{%
    \ProbabilityRender{\Psymb}{#1}%
}

\def\Ex#1{%
    \ProbabilityRender{\Esymb}{#1}%
}

\def\condPE#1#2{%
	\@ifnextchar\bgroup
	{\ConditionalProbabilityRender{\widetilde{\Esymb}}{#1}{#2}}
	{\ProbabilityRender{\widetilde{\Esymb}}{#1 \given #2}}
}

\def\Var#1{%
    \ProbabilityRender{\Varsymb}{#1}%
}

\def\ConditionalProbabilityRender#1#2#3#4{
	\renderwithdist{#1}{#2}{#3 \given #4}	
}

\def\ProbabilityRender#1#2{
  \@ifnextchar\bgroup%
  {\renderwithdist{#1}{#2}}
   {\singlervrender{#1}{#2}}
}
\def\singlervrender#1#2{%
   \ensuremath{\mathchoice
       {{#1}\left[ #2 \right]}
       {{#1}[ #2 ]}
       {{#1}[ #2 ]}
       {{#1}[ #2 ]}
   }
}
\def\renderwithdist#1#2#3{%
   \@ifnextchar\bgroup
   {\superfancyrender{#1}{#2}{#3}}
   {\ensuremath{\mathchoice
      {\underset{#2}{#1}\left[ #3 \right]}
      {{#1}_{#2}[ #3 ]}
      {{#1}_{#2}[ #3 ]}
      {{#1}_{#2}[ #3 ]}
     }
   }
}
\def\superfancyrender#1#2#3#4#5{
   \ensuremath{\mathchoice
      {\underset{#1}{{#1}}\left#4 #3 \right#5}
      {{#1}_{#2}#4 #3 #5}
      {{#1}_{#2}#4 #3 #5}
      {{#1}_{#2}#4 #3 #5}
   }
}
\makeatother

\newfont{\inhead}{eufm10 scaled\magstep1}

\newcommand{\calG}{{\cal G}}
\newcommand{\calH}{{\cal H}}

\newcommand{\calP}{{\cal P}}
\newcommand{\calO}{{\cal O}}

\newcommand{\calX}{{\cal X}}
\newcommand{\calT}{{\cal T}}

\newcommand{\calZ}{{\cal Z}}

\DeclareMathOperator{\tr}{\operatorname {tr}}

\newcommand{\inparen}[1]{\left(#1\right)}             
\newcommand{\inbraces}[1]{\left\{#1\right\}}           
\newcommand{\insquare}[1]{\left[#1\right]}             

\newcommand{\Poincare}{Poincar\'e\xspace}
\newcommand{\Erdos}{Erd\H{o}s\xspace}
\newcommand{\Renyi}{R\'enyi\xspace}



%% file: intro.tex
For fixed matrices $\mat{C_1}, \ldots, \mat{C_n}$ and  independent scalar random variables $x_1, \ldots, x_n$,
consider the problem of analyzing the random matrix
\[
\mat{M} ~=~ \mat{C_1} \cdot x_1 + \cdots + \mat{C_n} \cdot x_n \mper
\]
Note that the entries of the random matrix $\mat{M}$ are not necessarily independent, but are (possibly correlated) linear functions of the independent random variables $x_1, \ldots, x_n$. 
Such matrices, which arise in a variety of applications in algorithms, statistics, and numerical linear algebra, can often be shown as being concentrated around the mean, using a rich selection of matrix deviation inequalities~\cite{Tro15}. 
Moreover, when the random variables $x_1, \ldots, x_n$ are Gaussian, recent breakthrough results
have obtained even sharper concentration guarantees depending on the structure of the matrices
$\mat{C_1}, \ldots, \mat{C_n}$~\cite{BBMVH23, BVH22} already leading to applications in discrepancy
theory~\cite{BJM23} and quantum information theory~\cite{LY23}.

In contrast to the above random matrices which are linear functions of independent random variables, several recent applications in spectral algorithms and lower bounds for statistical problems require understanding the (expected) norms of random matrices of the form
\[
\mat{F} ~=~ \sum_{S \subseteq [n] \atop \abs{S} \leq d} \mat{C}_S \cdot \prod_{i \in S} x_i \mper
\]
The above random matrix, which denote as a matrix-valued function $\mat{F}(\bfx)$, is a (multilinear, in the example above) \emph{low-degree polynomial} in the vector of random variables $\bfx = (x_1, \ldots, x_n)$. 
Such matrix-valued polynomial functions arise naturally in a variety of applications, including
algorithms for tensor decomposition and completion~\cite{HSSS16, HSS19, KP20, DDOCHST22}, orbit
recovery~\cite{MW19}, power-sum decompositions and learning of Gaussian mixtures~\cite{BHKX22}, and
lower bounds for Sum-of-Squares hierarchies~\cite{AMP21, BHKKMP19, JPRTX21, R22:thesis}. 
In general, such matrices are often an important technical component in the analysis of spectral
algorithms, which require understanding the spectrum of some random matrix depending on data
obtained as independent samples. Given the expressive power of polynomials, one can often write the
matrix entries as low-degree polynomials in the data.

In several applications above, one is interested in obtaining bounds on the spectral norm of the
deviation matrix $\mat{F}(\bfx) - \ExpOp\mat{F}(\bfx)$, which hold with high probability over the choice of $\bfx$. 
Although matrix-deviation inequalities for nonlinear random matrices have been a subject of active research in recent work~\cite{PMT16, MJC12, ABY20, HT20, HT21}, the analyses in the applications above have often needed to rely on estimating matrix norms via direct computations of trace moments.
While these do yield sharp bounds required for applications, they require somewhat intricate computations and ingenious combinatorial arguments tailored to the applications at hand.

\paragraph{Analyzing concentration via moment estimates.}
We consider concentration bounds on a \emph{scalar} polynomial $f(\bfx)$ with mean zero. Using Markov's inequality, this can be reduced to computing moment estimates, since
\[
\Pr{\abs{f(\bfx)} \geq \lambda}
~=~ \Pr{\inparen{f(\bfx)}^{2t} \geq \lambda^{2t}}
~\leq~ \lambda^{-2t} \cdot {\Ex{\inparen{f(\bfx)}^{2t}}}
\]
Note that while in some cases $\Ex{\inparen{f(\bfx)}^{2t}}$ can be computed by direct expansion, it often involves an intricate analysis of the structure of terms with degrees growing with $t$, and therefore indirect methods may be more convenient.
Methods for obtaining concentration of (scalar) polynomial functions have been the subject of a large body of work, including results by Kim and Vu~\cite{KV00}, Lata{\l}a~\cite{Latala06}, Schudy and Sviridenko~\cite{SS11:multilinear, SS12:polynomials}, Adamczak and Wolff~\cite{AW15}, and Bobkov, G{\"o}tze and Sambale~\cite{BGS19}. 
A useful comparison to direct computation, is the method of Adamczak and Wolff, which is applicable to functions $f(\bfx)$ of a broad range of random vectors $\bfx$ (with not necessarily independent coordinates) from a distribution obeying the \Poincare inequality.
They obtain such moment bounds by recursive applications of the \Poincare inequality, reducing the computation of the moments of $f$ to that of its derivatives.
For a degree-$d$ polynomial function $f$, one can consider its coefficients as forming a (symmetric) order-$d$ tensor, and their results obtain bounds on the moments of $f$ in terms of the norms of various lower-order tensors ``flattening" of this coefficient tensor.
Thus, the problem of estimating moments of the random variable $f(\bfx)$, can be reduced in a
black-box way, to understanding the norms of a small number (depending on $d$) of
\emph{deterministic} tensors. Any dependence on the problem structure can then be limited to
understanding these norms, where often crude estimates can suffice. 

The matrix analog of the Markov argument involves the Schatten-$2t$ norm $\norm{.}_{2t}$, which is defined for a matrix $\mat{M} $ with non-zero singular values $\sigma_1, \ldots, \sigma_r$ as $\norm{\mat{M}}_{2t}^{2t} ~:=~ \sum_{j \in [r]} \sigma_j^{2t}$.
For a function $\mat{F}$ with $\Ex{\mat{F}(Z)} = \bf 0$, we have the following bound using Schatten norms.
\[
\Pr{\sigma_1(\mat{F}) \geq \lambda}
~\leq~ \lambda^{-2t} \cdot \Esch{\mF}{2t} 
~=~ \lambda^{-2t} \cdot \Etr{\inparen{\mF(\bfx)\mF(\bfx)^\T}^{t}}
\]
Several applications requiring norm bounds for matrix-valued polynomial functions, start with the
above inequality, and often rely on direct expansion of the trace for the matrix power
$\inparen{\mF(\bfx)\mF(\bfx)^\T}^{t}$.
Recall that when $\mA$ is the adjacency matrix of a graph, $\tr(\mA^{2t})$ can be interpreted as the number of cycles of length $2t$.
Similarly, when working with matrices $\mF(\bfx)$ where each entry can be interpreted as arising
from a combinatorial pattern, $\inparen{\mF(\bfx)\mF(\bfx)^\T}^{t}$ can be viewed as the number of
copies of such patterns ``chained" together in a certain way. Estimating the trace then amounts to
estimating the (expected) number of such chainings~\cite{AMP21}. 

Obtaining bounds on the $\Esch{\mF}{2t}$ via this method requires first using problem structure to
decompose $\mF$  into matrices with such combinatorial patterns, and then using (ingenious) combinatorial arguments to understand the expected number of occurrences of such chains of patterns.
The focus of this work is understanding alternative general methods for estimating such norm bounds, when the
underlying problem may not necessarily have a combinatorial flavor, or when such decompositions may
be hard to analyze.

\paragraph{Efron-Stein inequalities.}
Efron-Stein inequalities bound the global variance of a function of independent random variables, in terms of local variance estimates obtained by changing one variable at a time.
For $i \in [n]$ and tuple $\bfx = (x_1, \ldots, x_n)$, let $\bfx^{(i)}$ denote the tuple $(x_1, \ldots, x_{i-1}, \resamp{x_i}, x_{i+1}, \ldots, x_n)$, where $\resamp{x_i}$ is an independent copy of $x_i$.
For a scalar function $f(\bfx)$, the Efron-Stein inequality states that
\[
\Var{f(\bfx)} ~=~ \Ex{\inparen{f(\bfx) - \EE f}^2} ~\leq~ \frac12 \cdot \sum_{i \in [n]} \Ex{\inparen{f(\bfx) - f\inparen{\bfx^{(i)}}}^2} ~=~ \Ex{V(\bfx)}\mcom
\]
where $V(\bfx) ~:=~ \sum_{i \in [n]} \EE \insquare{(f(\bfx) - f(\bfx^{(i)}))^2 ~  \Big |  ~ \bfx ~}$.
%
%
Moment versions of the Efron-Stein inequality were developed by Boucheron \etal~\cite{BBLM05}, who
proved\footnote{In fact, the estimates of ~\cite{BBLM05} are in terms of more refined quantities
  $V_+(\bfx)$ and $V_{-}(\bfx)$, which are both upper bounded by $V(\bfx)$ but can be much smaller.}
\[
\Ex{\inparen{f(\bfx) - \EE f}^{2t}} ~\le~ (C_0 \cdot t)^t \cdot \Ex{\inparen{V(\bfx)}^{t}} \mper
\]
A beautiful matrix generalization of the above inequality was obtained by Paulin, Mackey and
Tropp~\cite{PMT16}, via the method of exchangeable pairs.
This was also extended in later works~\cite{ABY20,HT20} to settings where the random
input vector $\bfx$ does not necessarily come from a product distribution, but from a distribution
satisfying a Poincar{\'e} inequality.

Rajendran and Tulsiani~\cite{RT23}  used the matrix Efron-Stein inequality of Paulin, Mackey and
Tropp~\cite{PMT16}  to obtain norm bounds for matrix-valued polynomial functions $\mat{F}(\bfx)$ of
independent variables $x_1, \ldots, x_n$. 
For (say) a degree-$d$ homogeneous function $\mat{F}(\bfx)$ their method relied on iteratively
applying the matrix Efron-Stein inequality, and at each step interpreting the variance proxy
$\mat{V}(\bfx)$ as the square of a degree-$(d-1)$ polynomial matrix. 
While this method indeed yields simple bounds when $x_1, \ldots, x_n$ are independent Rademacher
random variables, this simple recursion based on the black-box use of the inequality
from~\cite{PMT16} unfortunately stalls for other product distributions. 
As a result, the proof and also the form of the result in ~\cite{RT23} was significantly more
involved for other product distributions (including Gaussians).
For the general case, the proof required modifying the exchangeable pairs argument of ~\cite{PMT16},
and the bounds reduced the question of estimating moments of matrix-valued functions to those of
scalar polynomials (which were then analyzed using known methods).

\paragraph{Decoupling.}
Decoupling inequalities were developed in the study of $U$-statistics \cite{PG:book}, multiple
stochastic integration \cite{MT84}, and polynomial chaos \cite{KW15}, and have found important
applications in applied mathematics, theoretical computer science, applied probability theory and
statistics. Some examples include the study of compressed sensing \cite{Rau10}, query complexity
\cite{DZ16}, the proof of Hanson-Wright inequality \cite{Ver18}, learning mixture of Gaussians
\cite{GHK15} and so on. 
Moreover, the inequalities are applicable in both scalar and matrix settings, and even more broadly
in Banach spaces. 

For a degree-$d$ homogeneous \emph{multilinear} polynomial $f(\bfx)$ in $n$ independent random
variables, standard decoupling inequalities (see \cref{sec:decoupling}) can be used to obtain
\[
\Ex{\bfx}{\inparen{f(\bfx)}^{2t}} ~\leq~ C_d^{2t} \cdot \Ex{\bfx^{(1)}, \ldots, \bfx^{(d)}}{\inparen{f(\bfx^{(1)}, \ldots, \bfx^{(d)})}^{2t}} \mcom
\]
where $f(\bfx^{(1)}, \ldots, \bfx^{(d)})$ denotes the polynomial in $d \cdot n$ random variables
obtained by replacing the $d$ variables in each (ordered) monomial by the corresponding $d$
coordinates coming from the $d$ independent random vectors $\bfx^{(1)}, \ldots, \bfx^{(d)}$, and
$C_d$ is a constant depending on the degree $d$. 
One can now fix the vectors $\bfx^{(1)}, \ldots, \bfx^{(d-1)}$, and consider the expectation 
\[
\Ex{\bfx^{(d)}}{\inparen{f(\bfx^{(1)}, \ldots, \bfx^{(d)})}^{2t}} ~=~ \Ex{\bfx^{(d)}}{\inparen{\sum_{i \in [n]} f_{i}(\bfx^{(1)}, \ldots, \bfx^{(d-1)}) \cdot x_i^{(d)}}^{2t}} \mcom
\] 
where we write $f(\bfx^{(1)}, \ldots, \bfx^{(d)})$ as a linear form in $\bfx^{(d)}$ with
coefficients $f_{i}(\bfx^{(1)}, \ldots, \bfx^{(d-1)})$.
The moments of such a linear form can be understood (for example) using the Khintchine inequality,
which bounds it in terms of a polynomial depending only on $\bfx^{(1)}, \ldots, \bfx^{(d-1)}$.

Similarly, in the matrix case, we use decoupling inequalities to reduce the problem of estimating
moments for degree-$d$ homogeneous (and multilinear) matrix-valued polynomials, to that of
estimating moments for \emph{linear} matrix-valued functions. 
At this point, one can apply standard and well-known linear matrix concentration inequalities (we
use the matrix Rosenthal inequality) which yield a bound in terms of a degree-$(d-1)$ matrix-valued
polynomial function of the vectors $\bfx^{(1)}, \ldots, \bfx^{(d-1)}$. 
Iterating this process gives a simple method for obtaining norm bounds, applicable for
\emph{multilinear} polynomial functions on $n$ independent random variables (with any reasonable
distribution).

While this method does not immediately apply for non-multilinear polynomials, it suffices for many
of the applications mentioned above. Moreover, for arbitrary polynomials in Gaussian random
variables, we can also obtain a simple recursion as an immediate consequence of the more recent
\Poincare inequalities of Huang and Tropp~\cite{HT21}. Together, these suffice for most
applications of interest.


\newcommand{\tuple}[2]{\calT^{#1}_{#2}}

\subsection*{Methods and results: a technical overview}
Let $x_1, \ldots, x_n$ be i.i.d real-valued random variables with $\Ex{x_i}=0, \Ex{x_i^2} = 1$ and
$\abs{x_i} \leq L$ for each $i \in [n]$. 
Let $\calT^d_n \subseteq [n]^d$ denote the set of ordered $d$-tuples $\bfi = (i_1, \ldots, i_d)$
with $i_1, \ldots, i_d$ all distinct.
For a fixed sequence of (deterministic) matrices $\inbraces{\mat{A}_{\bfi}}_{\bfi \in \tuple{d}{n}}
  \subseteq \R^{d_1 \times d_2}$, we consider a matrix-valued multilinear polynomial function
  defined as
\[
\mat{F}(\bfx) ~=~ \sum_{\bfi \in \tuple{d}{n}} \mat{A}_{\bfi} \cdot \prod_{j \in [d]} x_{i_j} \mper
\]
We write the polynomial in terms of ordered tuples for decoupling, but since the variables $x_1,
\ldots, x_n$ commute, we can assume that for any permutation $\sigma: [d] \to [d]$ permuting the $d$
coordinates of each tuple, $\mat{A}_{\bfi} = \mat{A}_{\sigma(\bfi)}$ (this is referred to as
\emph{permutation symmetry} of $\mat{F}$).
We can use the decoupling inequalities from \cref{sec:decoupling} to show that 
\[
\norm{\mat{F}(\bfx)}_{4t} 
~\leq~ C_d \cdot \norm{\mat{F}(\bfx^{(1)}, \ldots, \bfx^{(d)})}_{4t} 
~=~ C_d \cdot \norm{\sum_{\bfi \in \tuple{d}{n}} \mat{A}_{\bfi} \cdot \prod_{j \in [d]}
  x_{i_j}^{(j)}}_{4t} \mcom
\]
where $\norm{\cdot}_{4t}$ denotes the (expected) Schatten norm $\inparen{\Ex{\tr
    (\trans{\mat{F}(\bfx)}\mat{F}(\bfx))^{2t}}}^{1/4t}$.

We remark that while the constant $C_d$ is of the form $d^d$ in general, this can be improved when
the underlying set of indices $[n]$ has more structure. For example, when $[n]$ corresponds to the
set of \emph{pairs} in an base set $[r]$ and each of the monomials involves at most $k$ elements of
$[r]$ (has \emph{index degree} $k$), it is possible to improve the constant to $k^k$. 
This is essentially the ``vertex partitioning'' argument used in works on graph matrices, and is
discussed in \cref{sec:structured} in the language of decoupling.

As mentioned earlier, we can now ``linearize'' the function $\mat{F}(\bfx^{(1)}, \ldots,
\bfx^{(d)})$ by fixing $\bfx^{(d)}, \ldots, \bfx^{(d-1)}$ and treating it only as a function of
$\bfx^{(d)}$ \ie we consider
\begin{align*}
\ExpOp \norm{\mat{F}(\bfx^{(1)}, \ldots, \bfx^{(d)})}_{4t}^{4t}
~=~ 
\ExpOp \norm{\sum_{\bfi \in \tuple{d}{n}} \mat{A}_{\bfi} \cdot \prod_{j \in [d]}
  x_{i_j}^{(j)}}_{4t}^{4t}
&~=~ 
\ExpOp_{\bfx^{(1)}, \ldots, \bfx^{(d-1)}} ~~\ExpOp_{\bfx^{(d)}} \norm{\sum_{k \in [n]} \inparen{\sum_{\bfi \in \tuple{d}{n}
    \atop i_d = k} \mat{A}_{\bfi} \cdot \prod_{j \in [d-1]}
  x_{i_j}^{(j)}} \cdot x_k^{(d)}}_{4t}^{4t} \\
&~=~
\ExpOp_{\bfx^{(1)}, \ldots, \bfx^{(d-1)}} ~~\ExpOp_{\bfx^{(d)}} \norm{\sum_{k \in [n]}
  \inparen{\partial_{x_k^{(d)}}\mat{F}} \cdot x_k^{(d)}}_{4t}^{4t} \mper
\end{align*}
To bound the inner expectation, we can now use the matrix Rosenthal inequality (\cref{rosenthal})
which says that for a collection of centered, independent random matrices $\inbraces{\mat{Y}_k}$ with
finite moments, we have
\[
    \E \norm{\sum_{k}\bfY_k}_{4t}^{4t} \leq (16t)^{3t} \cdot \left\{ \norm{\left(\sum_k\E  \bfY_k \trans{\bfY}_k \right)^{1/2}}_{4t}^{4t} + \norm{ \left(\sum_k\E \trans{\bfY}_k\bfY_k \right)^{1/2} }_{4t}^{4t} \right\} + (8t)^{4t}\cdot \left(\sum_k \E \norm{\bfY_k}_{4t}^{4t} \right).
\]
Taking $\mat{Y}_k = \partial_{x_k^{(d)}}\mat{F} \cdot x_k^{(d)}$ we can now compute the
expectations (over $\bfx^{(d)}$) in the RHS as
\begin{align*}
\sum_k \E\mat{Y}_k \trans{\mat{Y}}_k &~=~   
\begin{bmatrix}
  \partial_{x_1^{(d)}} \mat{F} & \dots & \partial_{x_n^{(d)}} \mat{F}
\end{bmatrix}
\trans{
\begin{bmatrix}
  \partial_{x_1^{(d)}} \mat{F} & \dots & \partial_{x_n^{(d)}} \mat{F} 
\end{bmatrix}
} \\
\sum_k \E\trans{\mat{Y}}_k\mat{Y}_k  &~=~   
\begin{bmatrix}
  \partial_{x_1^{(d)}} \trans{\mat{F}} & \dots & \partial_{x_n^{(d)}} \trans{\mat{F}}
\end{bmatrix}
\trans{
\begin{bmatrix}
  \partial_{x_1^{(d)}} \trans{\mat{F}} & \dots & \partial_{x_n^{(d)}} \trans{\mat{F}} 
\end{bmatrix}
} \\
\sum_k \E \norm{\bfY_k}_{4t}^{4t} &~=~ 
\E \norm{\begin{bmatrix}
 x_1^{(d)} \cdot \partial_{x_1^{(d)}} \mat{F} & & \\ & \ddots & \\ & & x_n^{(d)} \cdot \partial_{x_n^{(d)}} \mat{F}
\end{bmatrix}}_{4t}^{4t}
~\leq~ L^{4t} \cdot \norm{\begin{bmatrix}
  \partial_{x_1^{(d)}} \mat{F} & & \\ & \ddots & \\ & & \partial_{x_n^{(d)}} \mat{F}
\end{bmatrix}}_{4t}^{4t}
\end{align*}
Iterating the above argument, we can prove a bound in terms of a collection of matrices
$\mat{F}_{a,b,c}$ constructed inductively as
\begin{align*}
 \bfF_{a+1,b,c} &~=~ \trans{\begin{bmatrix}
  \partial_{x_1} \bfF_{a,b,c} & \dots & \partial_{x_n} \bfF_{a,b,c}
\end{bmatrix}} \\
 \bfF_{a,b+1,c} &~=~ \begin{bmatrix}
  \partial_{x_1} \bfF_{a,b,c} & \dots & \partial_{x_n} \bfF_{a,b,c}
\end{bmatrix} \\
 \bfF_{a,b,c+1}&~=~ \begin{bmatrix}
  \partial_{x_1} \bfF_{a,b,c} & & \\ & \ddots & \\ & & \partial_{x_n} \bfF_{a,b,c}
\end{bmatrix} \mper
\end{align*}
Note that in the above definition, we do not specify the order of differentiation, which can be
ignored due to the permutation symmetry of $\mat{F}$ and the fact that changing the order does not
change the Schatten norms (see \cref{sec:prelims} and \cref{sec:multilinear} for details).
This argument yields the following bound for all homogeneous multilinear polynomial functions.
\begin{theorem} [Restatement of \cref{homo multilinear recursion}]
\label{multilinear:intro}
Let $\bfx = \{x_i\}_{i = 1}^n$ be a sequence of i.i.d random variables with $\E x_i= 0$,  $\E x_i^2= 1$ and $|x_i| \leq L$ for all $1\leq i \leq n$. Let $\{\bfA_{\bfi}\}_{\bfi \in \calT^d_n}$ be a multi-indexed sequence of deterministic matrices of the same dimension. Define a permutation symmetric, homogeneous multilinear polynomial random matrix of degree $d$ as
$$
\bfF(\bfx) =  \sum_{\bfi \in \calT^d_n} \left( \bfA_{\bfi} \cdot \prod_{j \in \{i_1,\dots, i_d \}} x_j\right).
$$
Let $a,b,c\in \Z_{\geq 0}$ and $d = a+b+c$. Then for $2 \leq t \leq \infty,$
$$
\E \norm{\bfF(\bfx)}_{4t}^{4t} ~\leq~ \sum_{\substack{a,b,c: \\ a+b+c=d}} (48dt)^{4dt} \cdot L^{4ct} \norm{\bfF_{a,b,c}}_{4t}^{4t}.
$$
\end{theorem}
We remark that the condition that $\abs{x_i} \leq L$ can be replaced by a condition on the growth of
moments of $x_i$, which is true for subgaussian variables.
The above bound is in terms of the norms of a constant number of \emph{deterministic} matrices
$\mat{F}_{a,b,c}$, as is also the case for general moment bounds on scalar polynomials~\cite{AW15}. 
We will see later that these deterministic matrices can easily be interpreted in several cases of interest, to
recover the bounds obtained via combinatorial methods.

\paragraph{Gaussian polynomial matrices.} While the above bounds are only for multilinear
polynomials, they can also be extended to \emph{arbitrary} polynomials of independent Gaussian
random variables, using standard techniques to approximate them by multilinear polynomials. 
However, for the case of Gaussian polynomials, the recent \Poincare inequalities of Huang and
Tropp~\cite{HT20} directly yield a simple bound, which is easier to apply. They show that for a
Hermitian matrix-valued function $\mat{H}$ of Gaussian valued random variables
\[
 \E \norm{\bfH(\bfx) - \E \bfH(\bfx)}_{2t}^{2t} ~\leq~ (\sqrt{2}t)^{2t} \cdot \E \tr \left(\sum_{i=1}^n (\partial_i \bfH(\bfx) )^2 \right)^t.
\]
Given a (not necessarily Hermitian) matrix-valued degree-$d$ polynomial function $\mat{F}(\bfx) \in
\R^{d_1 \times d_2}$, we can apply the above bound to the ``Hermitian dilation'' $\mat{H}
= \begin{bmatrix} \bf 0 & \mat{F} \\ \trans{\mat{F}} & \bf 0 \end{bmatrix}$, to get
\[
2\norm{\mat{F} - \E\mat{F}}_{2t}^{2t}
~\leq~
(\sqrt{2}t)^{2t} \cdot \norm{\left(\sum_{i=1}^n \partial_i \bfF ~ \partial_i \trans{\bfF}
  \right)^{1/2}}^{2t}_{2t} + (\sqrt{2}t)^{2t} \cdot \norm{\left(\sum_{i=1}^n \partial_i \trans{\bfF} ~ \partial_i \bfF \right)^{1/2}}_{2t}^{2t}
\]
As before, one can apply this argument inductively to obtain a bound in terms of the matrices
$\mat{F}_{a,b}$ which are defined similarly as the matrices $\mat{F}_{a,b,c}$ before, with $c = 0$.
Since we no longer have $\E \mat{F}_{a,b} = \mat{0}$ for $a+b < d$, the bound is in terms of the
expected matrices $\E\mat{F}_{a,b}$ for all $a,b$ with $a+b \leq d$.
\begin{theorem}[Restatement of \cref{gaussian recursion}] 
 Let $\bfx \sim \gauss{0}{\mathbb{I}_n}$ and $\{\bfA_{\bfi}\}_{\bfi \in [n]^d}$ be a sequence of deterministic matrices of the same dimension. Define a degree-$d$ homogeneous Gaussian polynomial random matrix as
$$
\bfF(\bfx) =  \sum_{\bfi \in [n]^d} \left( \bfA_{\bfi} \cdot \prod_{j \in \{i_1,\dots, i_d \}} x_j\right).
$$
Let $a,b \in \Z_{\geq 0}$. Then for $2 \leq t \leq \infty,$
$$
\E \norm{\bfF(\bfx) - \E \bfF(\bfx)}_{2t}^{2t} ~\leq~ (2^d \sqrt{2}t)^{2t} \left( \sum_{1 \leq a+b
    \leq d}\norm{\E \bfF_{a,b}}_{2t}^{2t} \right).
$$   
\end{theorem}
While the above theorem is stated for homogeneous polynomials, the proof is identical also for the
non-homogeneous ones, with the sum in the definition being over all tuple sizes.

\subsection*{Do try this at home: applying the framework}
We now discuss how to interpret the matrices $\mat{F}_{a,b,c}$ arising in the bound in
\cref{multilinear:intro} to recover the norm bounds for graph matrices derived via combinatorial
methods. This is discussed with more formal details in \cref{sec:applications}, but we present an
intuitive (at least for the authors) version of the argument here. 

Graph matrices are defined using constant-sized template ``shapes'' and provide a convenient basis
for expressing (large) random matrices where the the entries are low-degree polynomials in the
(normalized) indicators of edges in a $G_{n,p}$ random graph. Such matrices and their norm bounds
have been used for a large number of applications~\cite{AMP21}.

Let $N = \binom{n}{2}$ and fix a canonical bijection between the space $[N]$ and $\binom{[n]}{2}$,
and let $[N]$ index the space of all possible edges in a random graph on $n$ vertices.
For each $e = \{i,j\} \in \binom{[n]}{2}$, for $\Delta_p = \sqrt{(1-p)/p}$ let $x_{e}$ be independently
$1/\Delta_p$ with probability $(1-p)$ and $-\Delta_p$ with probability $p$, so that $\E x_e = 0$,
$\E x_e^2 = 1$, and $\abs{x_e} \leq \Delta_p$.

Let $(V(\tau), E(\tau))$ be a graph on $k$ vertices with $d$ edges, and let $U_{\tau}, V_{\tau}$ be
\emph{ordered} subsets of $V(\tau)$ of size $k_1$ and $k_2$ respectively. 
The tuple $\tau = (V(\tau), E(\tau), U_{\tau}, V_{\tau})$ is referred to as a ``shape'' and is used
to define the following graph matrix $\mat{M}_{\tau}$ with rows and columns indexed by
$\tuple{k_1}{n}$ and $\tuple{k_2}{n}$ respectively
\begin{align*}
\mat{M}_{\tau}[\bfi, \bfj] ~:=~ \sum_{\psi \in \tuple{k}{n}} \indicator{\psi(U_{\tau}) = \bfi} \cdot
\indicator{\psi(V_{\tau}) = \bfj} \cdot \prod_{e_0 \in E(\tau)} x_{\psi(e_0)} \mcom
\end{align*}
where we interpret a tuple $\psi \in \tuple{k}{n}$ as an injective function $\psi: V(\tau) \to [n]$
in the canonical way. 
If the function $\psi$ maps $U_{\tau}$ to $\bfi$ (the row-index) and $V_{\tau}$ to $\bfj$ (the column
index), we add the monomial to $\mat{M}_{\tau}[\bfi, \bfj]$ corresponding to the product of the
images of all edges in the ``pattern'' $E(\tau)$ under $\psi$.
Note that each nonzero entry of the matrix-valued function $\mat{M}$ is a
multilinear homogeneous polynomial in the variables $\inbraces{x_e}_{e \in \binom{[n]}{2}}$ with
degree $d = \abs{E(\tau)}$. 
We denote $\mat{M}_{\tau}$ by $\mat{F}$ for ease of notation in the discussion below.

\begin{figure}[!h]
	\centering
	\includegraphics[scale=0.25]{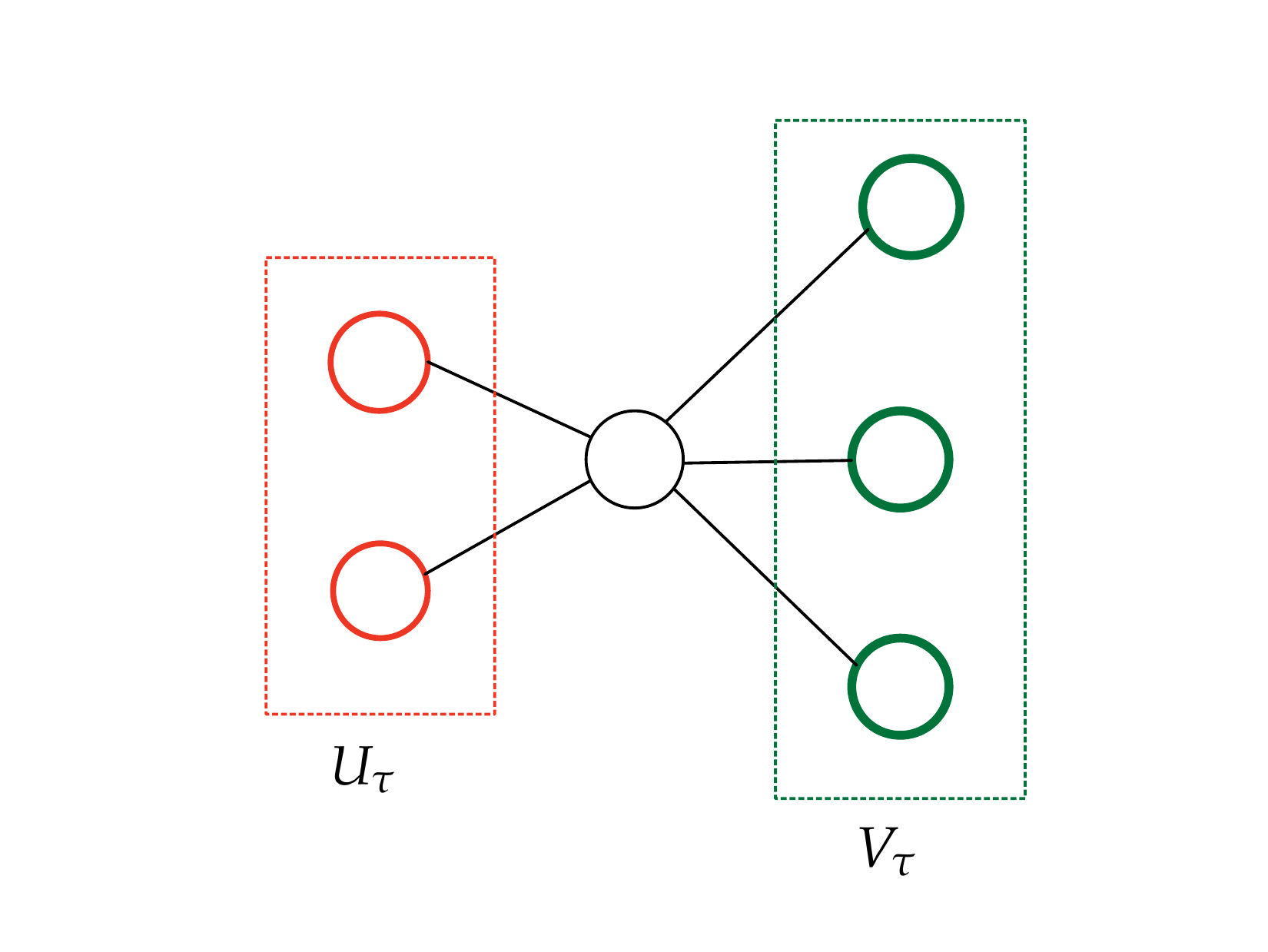}
	\caption{A shape $\tau$}
	\label{fig:shape}
\end{figure}

As will be apparent from the analysis below, the norm bounds for such matrices behave different
characterizations in the ``dense'' regime where $p = \Omega(1)$ and the ``sparse'' one where $p =
o(1)$. This is because the subgaussian norm $L = \Delta_p \approx p^{-1/2}$ is bounded in the first
case, and growing in the second case.
The bounds for the first case are stated in terms of the size the minimum vertex separator
separating $U_{\tau}$ from $V_{\tau}$ in the shape $\tau$. 
For the second case, one needs to augment the defintion of minimum vertex separator, to use a cost
based on the subgaussian norm $\Delta_p$.
We will show how to recover both these characterizations, using the matrices $\mat{F}_{a, b,
  c}$ given by \cref{multilinear:intro}.

\paragraph{Dense graph matrices.} We start by considering $p = 1/2$ so that the random variables
$x_e$ are just independent Rademacher variables. Taking $d$ to be constant, the bound in
\cref{multilinear:intro} involves only a constant number of matrices, and will depend on the
dominant term. 
To understand each of the terms, we first consider the matrices derived after one step of the
recursion. 
\begin{align*}
 \bfF_{1,0,0} ~=~ \trans{\begin{bmatrix}
  \partial_{x_{1}} \bfF & \dots & \partial_{x_{N}} \bfF
\end{bmatrix}},~~ 
 \bfF_{0,1,0} ~=~ \begin{bmatrix}
  \partial_{x_{1}} \bfF & \dots & \partial_{x_{N}} \bfF
\end{bmatrix} ~~\text{and}~~
 \bfF_{0,0,1} ~=~ \begin{bmatrix}
  \partial_{x_{1}} \bfF & & \\ & \ddots & \\ & & \partial_{x_{N}} \bfF
\end{bmatrix} \mper
\end{align*}
Since $\tr{\inparen{\sum_i \bfA_i \trans{\bfA}_i}^{2t}} + \tr{\inparen{\sum_i \trans{\bfA}_i
    \bfA_i}^{2t}} ~\geq~ \sum_i \tr\inparen{\bfA_i\trans{\bfA}_i}^{2t}$ and $L = 1$, we will have
that $\norm{\bfF_{1,0,0}}_{4t}^{4t} + \norm{\bfF_{0,1,0}}_{4t}^{4t} \geq L^{4t} \cdot
\norm{\bfF_{0,0,1}}_{4t}^{4t}$. 
This will also be true (up to constant factors) for any constant $L$, since we interested in the
$4t$-th root of the above trace. Using a similar argument, we can ignore terms with $c > 0$ for now
(we will come back to these in the sparse case) and focus on matrices $\mat{F}_{a,b,0}$.

We now interpret the matrices $\mat{F}_{1,0,0}$ and $\mat{F}_{0,1,0}$. The row-space of
$\mat{F}_{1,0,0}$ is now indexed by $(\bfi, e)$ for $e \in \binom{[n]}{2}$ (respectively $(\bfj, e)$
for the column space of $\mat{F}_{0,1,0}$). 
We now include terms corresponding to maps $\psi$ for which $\psi(U_{\tau}) = \bfi$, $\psi(V_{\tau})
= \bfj$ and $\psi(e_0) = e$ for some $e_0 \in E(\tau)$ so that $\partial x_e \mat{F}[\bfi, \bfj]
\neq 0$. 

Decomposing $\mat{F}_{1,0,0}$ into $\abs{E(\tau)}$ matrices (one for each choice of $e_0$), we can
consider the $\mat{F}_{1,0,0, e_0}$ as simply a new graph matrix we have \emph{deleted} the edge $e_0$ and
included both its end points in $U_{\tau}$ (respectively in $V_{\tau}$ for $\mat{F}_{0,1,0}$) to
obtain a new shape $\tau'$. 
Similarly, each $\mat{F}_{a,b,0}$ can be decomposed into $\abs{E(\tau)}^{a+b}$ graph matrices, where
we delete $a+b$ edges in total, and include the endpoints in $U_{\tau}$ for $a$ of them and
$V_{\tau}$ for $b$ of them.
This is illustrated in \cref{fig:evolution} where we color the edges red or green depending on their
inclusion in $U_{\tau}$ or $V_{\tau}$. 

\begin{figure}[!h]
	\centering 
	\includegraphics[scale=0.3]{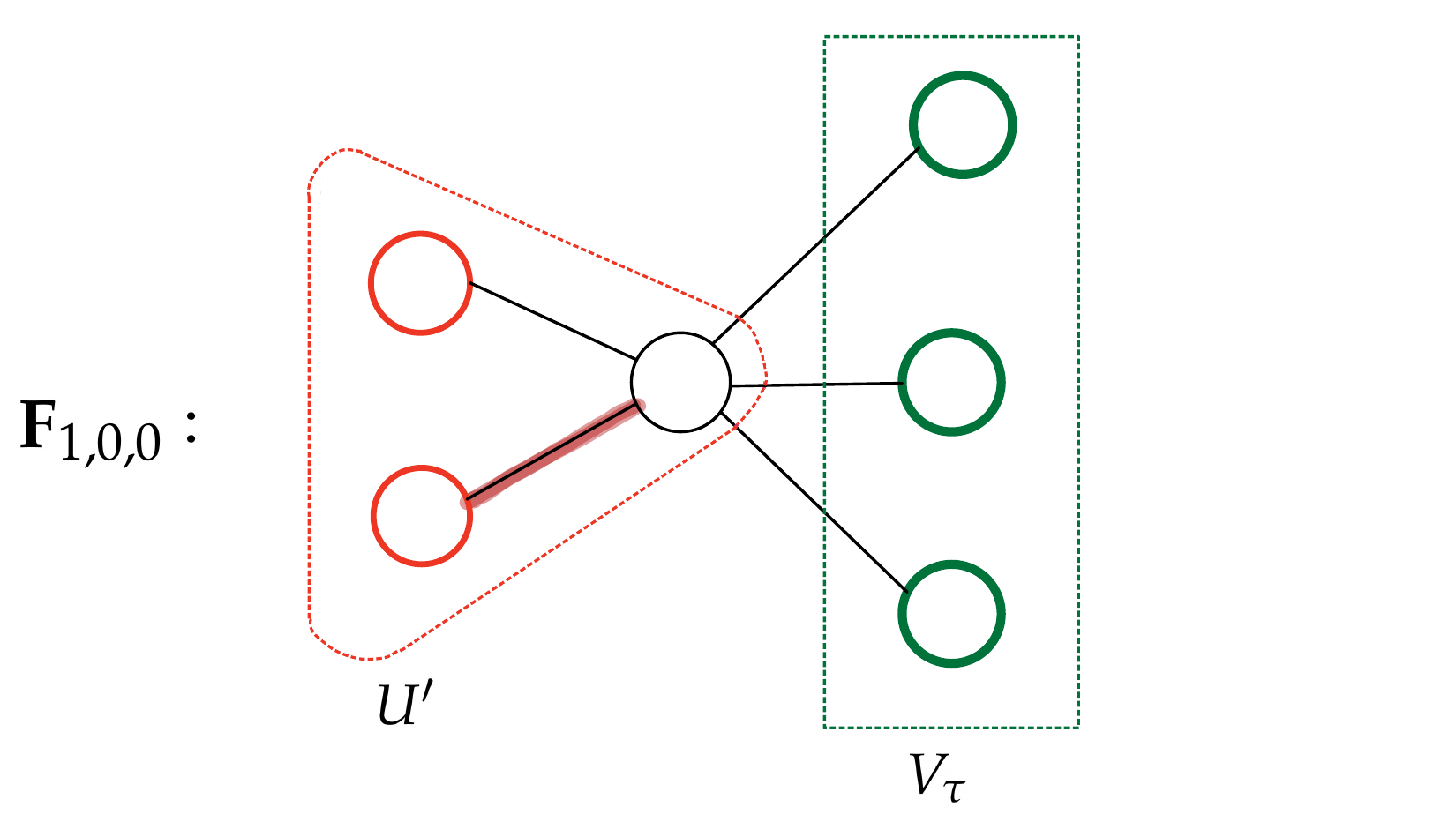}
	\includegraphics[scale=0.28]{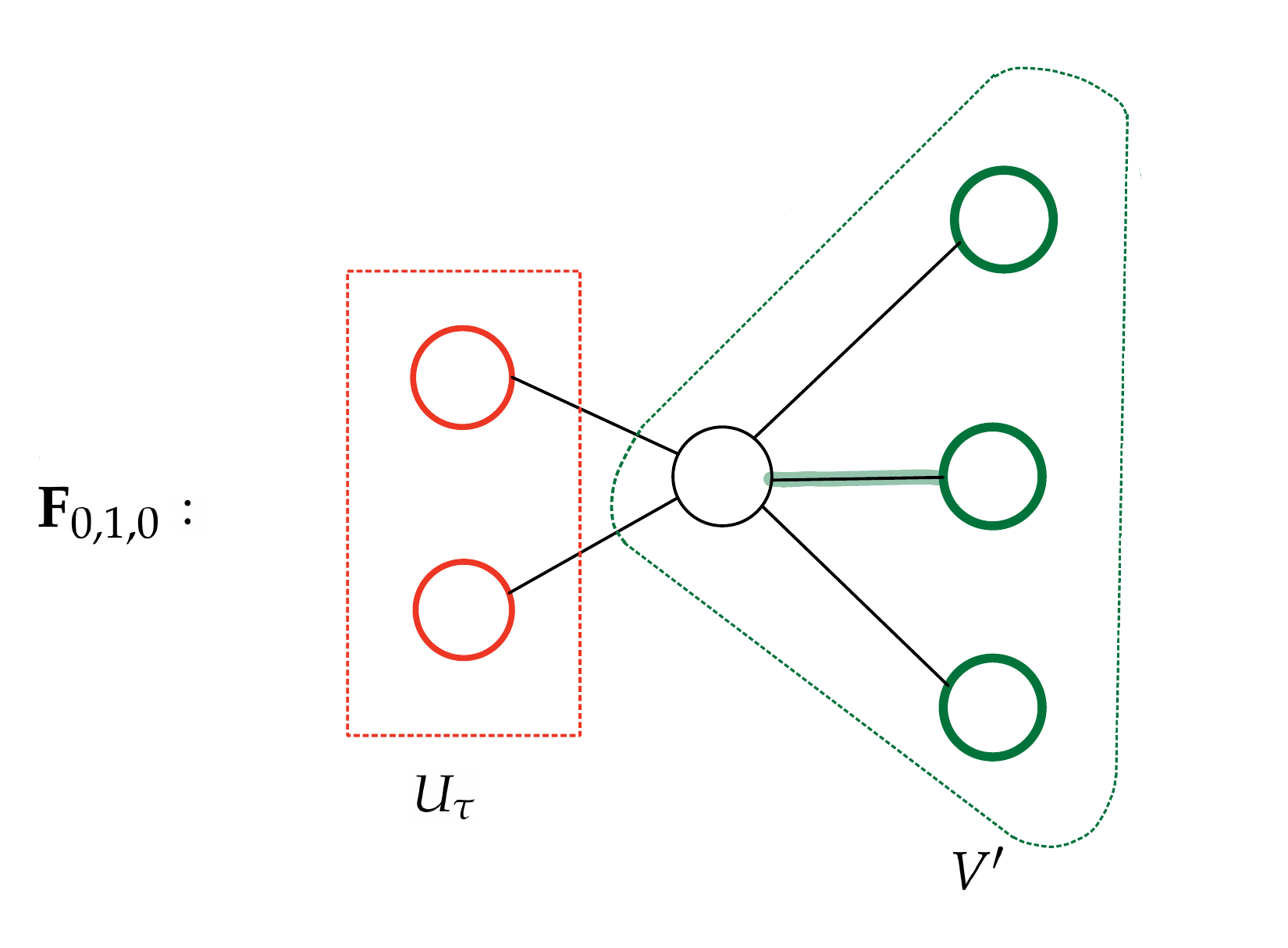}
	\includegraphics[scale=0.32]{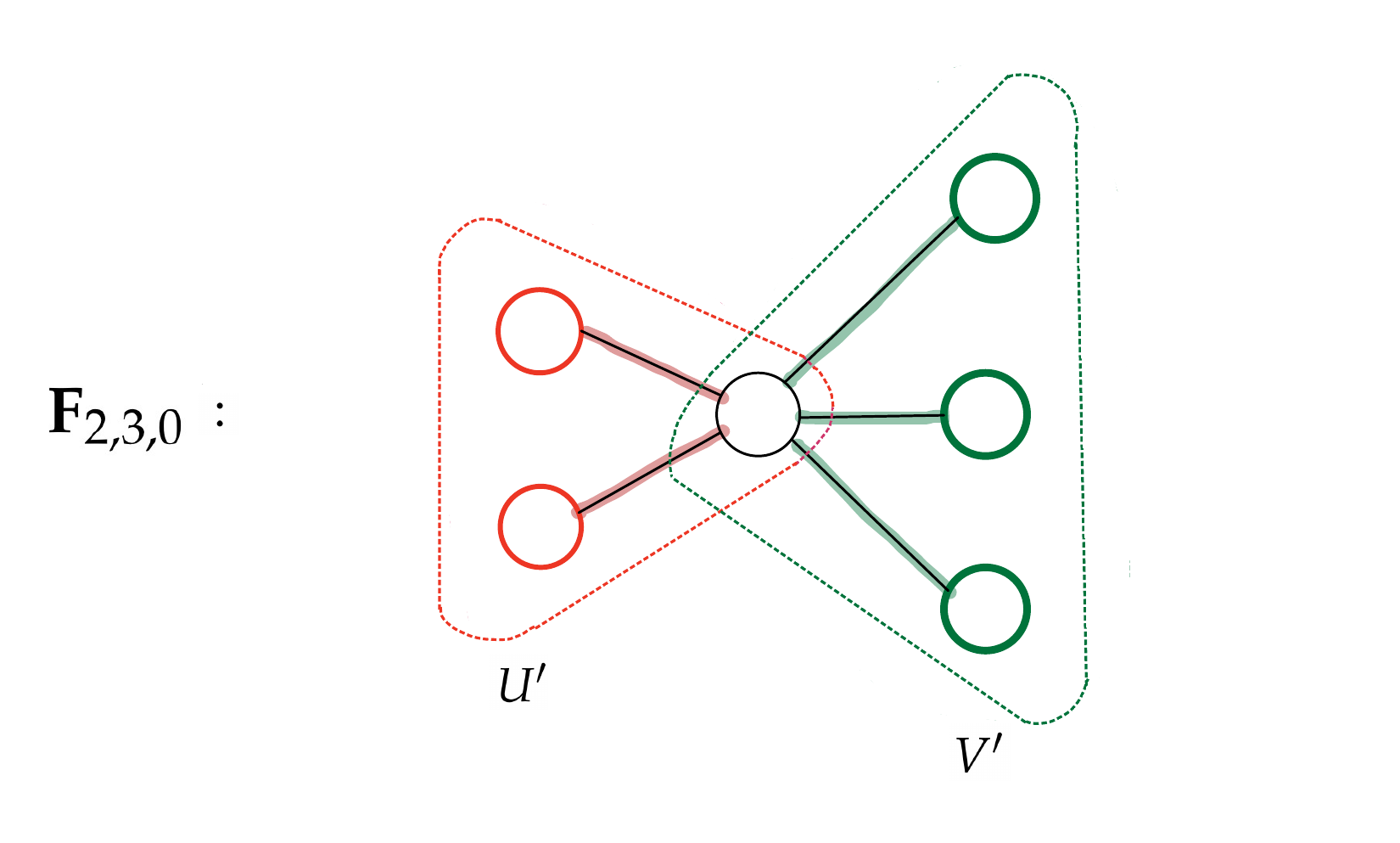}
	\caption{$\mat{F}_{1,0,0}$, $\mat{F}_{0,1,0}$ and $\mat{F}_{2,3,0}$}
	\label{fig:evolution}
\end{figure}

The bound in \cref{multilinear:intro} is then in terms of these new graph matrices where we delete
all $d$ edges, and split their endpoints between $U_{\tau}$ or $V_{\tau}$ to obtain $U', V'
\subseteq V(\tau) = [k]$. For an entry $\mat{M}[\bfi', \bfj']$ of any such matrix, 
there is at most one $\psi: V(\tau) \to [n]$ such
that ${\psi(U') = \bfi'}$ and $\psi (V') = \bfj'$ (assuming $\tau$ has no isolated vertices), and 
$\mat{M}[\bfi', \bfj'] = \indicator{\psi(U') = \bfi'} \cdot \indicator{\psi (V') = \bfj'}$. Taking $S =
U' \cap V'$, this matrix can be simply written as a block-diagonal matrix with $n^{\abs{S}}$ of size
$n^{\abs{U' \setminus S}} \times n^{\abs{V' \setminus S}}$. 
It is easy to check that the spectral norm of such a matrix is $n^{(\abs{U' \setminus S} + \abs{V'
    \setminus S})/2} = n^{(k - |S|)/2}$ since $\abs{U' \cup V'} = k$ and $S = U' \cap V'$.

Thus, each of the terms in the bound will correspond to graph matrices with different $U'$ and $V'$
(obtained by coloring edges red or green according to the above process) and the dominant term will
be the one with the smallest value of $\abs{U' \cap V'}$. We now claim that $U' \cap V'$ must be a
vertex separator separating $U_{\tau}$ and $V_{\tau}$ in the shape $\tau$ \ie any path between them
must pass through $U' \cap V'$. 
If the path contains both red and green edges, then it contains one vertex with both red and green
edges incident on it, which is then in $U' \cap V'$ (since $U'$ is obtain by adding endpoint of
red edges to $U_{\tau}$ and $V'$ similarly for green edges). If path is entirely red, then we have
an endpoint of a red edge in $V_{\tau}$ which is a vertex in $U' \cap V'$, and similarly for green paths.

Thus, we have $n^{(k - \abs{S})/2} \leq n^{(k-r)/2}$, where $r$ is the size of the minimum vertex
separator between $U_\tau$ and $V_{\tau}$, which recovers the known bound for dense graph
matrices~\cite{AMP21}.

\paragraph{Sparse graph matrices.} The argument for sparse graph matrices is almost the same as
above except now we also need to consider terms of the form $\mat{F}_{a,b,c}$ for $c > 0$. Just as
we interpreted $\mat{F}_{a,b,0}$ as coloring $a$ edges from $E(\tau)$ as red and including in $U'
\supseteq U_{\tau}$ and $b$ green edges in $V' \supseteq V_{\tau}$, we can now interpret
$\mat{F}_{a,b,c}$ as having $c$ edges whose endpoints are included in \emph{both} $U'$ and $V'$ (say
these edges are colored yellow). This simply reflects the fact that the derivatives in the
definition of $\mat{F}_{a,b,c+1}$ are placed as diagonal blocks.

\begin{figure}[!h]
	\centering
	\includegraphics[scale=0.3]{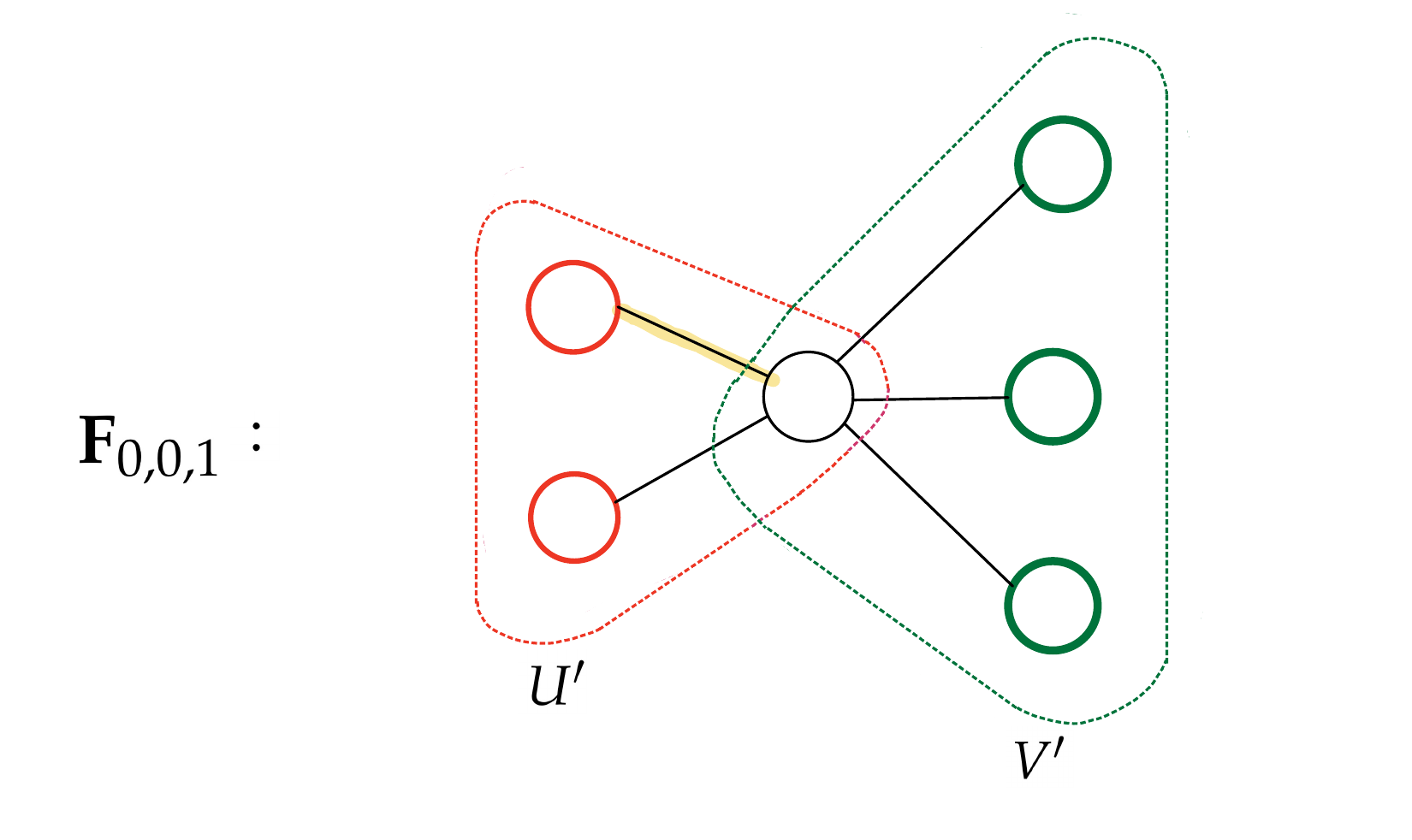}
	\caption{$\mat{F}_{0,0,1}$}
	\label{fig:sparse-evol}
\end{figure}

As we saw above, increasing the intersection of $U'$ and $V'$ decreases the norm, and so the
matrices with $c > 0$ should have a smaller Schatten norm. However, they are now included in the
bound with a multiplicative factor of $L^c$. Thus, we simply look for a vertex separator $S$ maximizing
$L^{e(S)} \cdot n^{(k - \abs{S})/2}$ where $e(S)$ counts the (yellow) edges contained in $S = U'
\cap V'$.  This is precisely the ``sparse vertex separator'' which determines the norm bound for
such sparse graph matrices~\cite{JPRTX21, RT23}.

%% file: decoupling-bounds.bbl
\newcommand{\etalchar}[1]{$^{#1}$}
\begin{thebibliography}{BHK{\etalchar{+}}19}

\bibitem[ABY20]{ABY20}
Richard Aoun, Marwa Banna, and Pierre Youssef.
\newblock Matrix poincar{\'e} inequalities and concentration.
\newblock {\em Advances in Mathematics}, 371:107251, 2020.

\bibitem[AMP21]{AMP21}
Kwangjun Ahn, Dhruv Medarametla, and Aaron Potechin.
\newblock Graph matrices: Norm bounds and applications.
\newblock {\em arXiv preprint https://arxiv.org/abs/1604.03423}, 2021.

\bibitem[AW15]{AW15}
Rados{\l}aw Adamczak and Pawe{\l} Wolff.
\newblock Concentration inequalities for non-lipschitz functions with bounded
  derivatives of higher order.
\newblock {\em Probability Theory and Related Fields}, 162(3):531--586, 2015.

\bibitem[BBLM05]{BBLM05}
St{\'e}phane Boucheron, Olivier Bousquet, G{\'a}bor Lugosi, and Pascal Massart.
\newblock {Moment inequalities for functions of independent random variables}.
\newblock {\em The Annals of Probability}, 33(2):514 -- 560, 2005.

\bibitem[BBvH23]{BBMVH23}
Afonso~S Bandeira, March~T Boedihardjo, and Ramon van Handel.
\newblock Matrix concentration inequalities and free probability.
\newblock {\em Inventiones {M}athematicae}, pages 1--69, 2023.

\bibitem[BGS19]{BGS19}
Sergey~G Bobkov, Friedrich G{\"o}tze, and Holger Sambale.
\newblock Higher order concentration of measure.
\newblock {\em Communications in Contemporary Mathematics}, 21(03):1850043,
  2019.

\bibitem[BHK{\etalchar{+}}19]{BHKKMP19}
Boaz Barak, Samuel Hopkins, Jonathan Kelner, Pravesh~K Kothari, Ankur Moitra,
  and Aaron Potechin.
\newblock A nearly tight sum-of-squares lower bound for the planted clique
  problem.
\newblock {\em SIAM Journal on Computing}, 48(2):687--735, 2019.

\bibitem[BHKX22]{BHKX22}
Mitali Bafna, Jun-Ting Hsieh, Pravesh~K. Kothari, and Jeff Xu.
\newblock Polynomial-time power-sum decomposition of polynomials.
\newblock In {\em 2022 IEEE 63rd Annual Symposium on Foundations of Computer
  Science (FOCS)}, pages 956--967, 2022.

\bibitem[BJM23]{BJM23}
Nikhil Bansal, Haotian Jiang, and Raghu Meka.
\newblock Resolving matrix {S}pencer conjecture up to poly-logarithmic rank.
\newblock In {\em Proceedings of the 55th ACM Symposium on Theory of
  Computing}, pages 1814--1819, 2023.

\bibitem[BvH22]{BVH22}
Tatiana Brailovskaya and Ramon van Handel.
\newblock Universality and sharp matrix concentration inequalities.
\newblock {\em arXiv preprint arXiv:2201.05142}, 2022.

\bibitem[DdL{\etalchar{+}}22]{DDOCHST22}
Jingqiu Ding, Tommaso d’Orsi, Chih-Hung Liu, David Steurer, and Stefan
  Tiegel.
\newblock Fast algorithm for overcomplete order-3 tensor decomposition.
\newblock In {\em Conference on Learning Theory}, pages 3741--3799. PMLR, 2022.

\bibitem[GHK15]{GHK15}
Rong Ge, Qingqing Huang, and Sham~M. Kakade.
\newblock Learning mixtures of gaussians in high dimensions.
\newblock In {\em Proceedings of the Forty-Seventh Annual ACM Symposium on
  Theory of Computing}, STOC '15, page 761–770, New York, NY, USA, 2015.
  Association for Computing Machinery.

\bibitem[Gur16]{Gur16}
Răzvan~Gheorghe Gurău.
\newblock {\em {Random Tensors}}.
\newblock Oxford University Press, 10 2016.

\bibitem[HSS15]{HSS15}
Samuel~B Hopkins, Jonathan Shi, and David Steurer.
\newblock Tensor principal component analysis via sum-of-square proofs.
\newblock In {\em Conference on Learning Theory}, pages 956--1006. PMLR, 2015.

\bibitem[HSS19]{HSS19}
Samuel~B Hopkins, Tselil Schramm, and Jonathan Shi.
\newblock A robust spectral algorithm for overcomplete tensor decomposition.
\newblock In {\em Conference on Learning Theory}, pages 1683--1722. PMLR, 2019.

\bibitem[HSSS16]{HSSS16}
Samuel~B Hopkins, Tselil Schramm, Jonathan Shi, and David Steurer.
\newblock Fast spectral algorithms from sum-of-squares proofs: tensor
  decomposition and planted sparse vectors.
\newblock In {\em Proceedings of the 48th ACM Symposium on Theory of
  Computing}, pages 178--191, 2016.

\bibitem[HT20]{HT20}
De~Huang and Joel~A. Tropp.
\newblock From poincar{\'e} inequalities to nonlinear matrix concentration.
\newblock {\em Bernoulli}, 2020.

\bibitem[HT21]{HT21}
De~Huang and Joel~A. Tropp.
\newblock Nonlinear matrix concentration via semigroup methods.
\newblock {\em Electronic Journal of Probability}, 26:Art. No. 8, Jan 2021.

\bibitem[JPR{\etalchar{+}}21]{JPRTX21}
Chris Jones, Aaron Potechin, Goutham Rajendran, Madhur Tulsiani, and Jeff Xu.
\newblock Sum-of-squares lower bounds for sparse independent set.
\newblock In {\em Proceedings of the 62nd IEEE Symposium on Foundations of
  Computer Science}, 2021.

\bibitem[KP20]{KP20}
Bohdan Kivva and Aaron Potechin.
\newblock Exact nuclear norm, completion and decomposition for random
  overcomplete tensors via degree-4 sos.
\newblock {\em arXiv preprint arXiv:2011.09416}, 2020.

\bibitem[KV00]{KV00}
Jeong~Han Kim and Van~H Vu.
\newblock Concentration of multivariate polynomials and its applications.
\newblock {\em Combinatorica}, 20(3):417--434, 2000.

\bibitem[Kwa87]{KW15}
Stanislaw Kwapien.
\newblock {Decoupling Inequalities for Polynomial Chaos}.
\newblock {\em The Annals of Probability}, 15(3):1062 --1071, 1987.

\bibitem[Lat06]{Latala06}
Rafa{\l} Lata{\l}a.
\newblock Estimates of moments and tails of gaussian chaoses.
\newblock {\em The Annals of Probability}, 34(6):2315--2331, 2006.

\bibitem[LY23]{LY23}
Cécilia Lancien and Pierre Youssef.
\newblock A note on quantum expanders, 2023.
\newblock \href {http://arxiv.org/abs/2302.07772} {\path{arXiv:2302.07772}}.

\bibitem[MJC{\etalchar{+}}12]{MJC12}
Lester Mackey, Michael Jordan, Richard Chen, Brendan Farrell, and Joel Tropp.
\newblock Matrix concentration inequalities via the method of exchangeable
  pairs.
\newblock {\em The Annals of Probability}, 42, 2012.

\bibitem[MT84]{MT84}
Terry~R. McConnell and Murad~S. Taqqu.
\newblock Double integration with respect to symmetric stable processes.
\newblock Technical report, Technical Report 618, Cornell Univ., 1984.

\bibitem[MW19]{MW19}
Ankur Moitra and Alexander~S Wein.
\newblock Spectral methods from tensor networks.
\newblock In {\em Proceedings of the 51st ACM Symposium on Theory of
  Computing}, pages 926--937, 2019.

\bibitem[OTR22]{OTR22}
Mohamed Ouerfelli, Mohamed Tamaazousti, and V.~Rivasseau.
\newblock Random tensor theory for tensor decomposition.
\newblock In {\em AAAI Conference on Artificial Intelligence}, 2022.

\bibitem[Oue22]{Oue22:thesis}
Mohamed Ouerfelli.
\newblock {\em New perspectives and tools for Tensor Principal Component
  Analysis and beyond}.
\newblock PhD thesis, Université Paris-Saclay, 2022.

\bibitem[OZ16]{DZ16}
Ryan O'Donnell and Yu~Zhao.
\newblock Polynomial bounds for decoupling, with applications.
\newblock In {\em Proceedings of the 31st Conference on Computational
  Complexity}, CCC '16, Dagstuhl, DEU, 2016. Schloss Dagstuhl--Leibniz-Zentrum
  fuer Informatik.

\bibitem[PG99]{PG:book}
Víctor~H. Peña and Evarist Giné.
\newblock {\em Decoupling: From dependence to independence}.
\newblock Springer-Verlag, 1999.

\bibitem[PMS95]{PMS95}
Victor~H. Pena and S.~J. Montgomery-Smith.
\newblock {Decoupling Inequalities for the Tail Probabilities of Multivariate
  $U$-Statistics}.
\newblock {\em The Annals of Probability}, 23(2):806 -- 816, 1995.

\bibitem[PMT16]{PMT16}
Daniel Paulin, Lester Mackey, and Joel~A. Tropp.
\newblock {Efron–Stein inequalities for random matrices}.
\newblock {\em The Annals of Probability}, 44(5):3431 -- 3473, 2016.

\bibitem[Raj22]{R22:thesis}
Goutham Rajendran.
\newblock {\em Nonlinear {R}andom {M}atrices and {A}pplications to the {S}um of
  {S}quares {H}ierarchy}.
\newblock PhD thesis, University of Chicago, 2022.

\bibitem[Rau10]{Rau10}
Holger Rauhut.
\newblock Compressive sensing and structured random matrices.
\newblock In {\em Theoretical Foundations and Numerical Methods for Sparse
  Recovery}, pages 1--92. De Gruyter, Berlin, New York, 2010.

\bibitem[RM14]{RM14}
Emile Richard and Andrea Montanari.
\newblock A statistical model for tensor pca.
\newblock In {\em Neural Information Processing Systems}, 2014.

\bibitem[RT23]{RT23}
Goutham Rajendran and Madhur Tulsiani.
\newblock Concentration of polynomial random matrices via efron-stein
  inequalities.
\newblock {\em Proceedings of the 2023 Annual ACM-SIAM Symposium on Discrete
  Algorithms (SODA)}, 2023.

\bibitem[SA13]{SA13}
Khalid Shebrawi and Hussien Albadawi.
\newblock Trace inequalities for matrices.
\newblock {\em Bulletin of the Australian Mathematical Society},
  87(1):139–148, 2013.

\bibitem[SS11]{SS11:multilinear}
Warren Schudy and Maxim Sviridenko.
\newblock Bernstein-like concentration and moment inequalities for polynomials
  of independent random variables: multilinear case.
\newblock {\em arXiv preprint arXiv:1109.5193}, 2011.

\bibitem[SS12]{SS12:polynomials}
Warren Schudy and Maxim Sviridenko.
\newblock Concentration and moment inequalities for polynomials of independent
  random variables.
\newblock In {\em Proceedings of the twenty-third annual ACM-SIAM symposium on
  Discrete Algorithms}, pages 437--446. SIAM, 2012.

\bibitem[Tro15]{Tro15}
Joel~A. Tropp.
\newblock An introduction to matrix concentration inequalities.
\newblock {\em Foundations and Trends in Machine Learning}, 8(1-2):1--230,
  2015.

\bibitem[Ver18]{Ver18}
Roman Vershynin.
\newblock {\em High-Dimensional Probability: An Introduction with Applications
  in Data Science}.
\newblock Cambridge University Press, 2018.

\end{thebibliography}
